\documentclass[12pt]{amsart}
\usepackage{graphicx,color}
\usepackage{colordvi}
\usepackage{amssymb, mathrsfs, amsfonts, amsmath}
\usepackage{latexsym}

\newtheorem{theorem}{Theorem}[section]
\newtheorem{lemma}[theorem]{Lemma}
\newtheorem{corollary}[theorem]{Corollary}
\newtheorem{proposition}[theorem]{Proposition}
\newtheorem{definition}[theorem]{Definition}
\newtheorem{example}[theorem]{Example}
\newtheorem{remark}[theorem]{Remark}

\newtheorem{conjecture}[theorem]{Conjecture}

\newcommand{\R}{\mathbb R}
\newcommand{\Q}{\mathbb Q}

\newcommand{\F}{\mathbb F}

\newcommand{\T}{\mathcal{T}}

\begin{document}

\title[Lipschitz Geometry of Pairs of Normally Embedded H\"older triangles
 ]{Lipschitz Geometry of Pairs of Normally Embedded H\"older triangles}

\author[]{Lev Birbrair*}\thanks{*Research supported under CNPq 302655/2014-0 grant and by Capes-Cofecub}
\address{Departamento de Matem\'atica, Universidade Federal do Cear\'a
(UFC), Campus do Pici, Bloco 914, Cep. 60455-760. Fortaleza-Ce,
Brasil and Institute of Mathematics, Jagiellonian University, Profesora Stanisława Łojasiewicza 6, 30-348 Kraków, Poland  } \email{lev.birbrair@gmail.com}

\author[]{Andrei Gabrielov}
\address{Department of Mathematics, Purdue University,
West Lafayette, IN 47907, USA}\email{gabriea@purdue.edu}

\date{\today}

\maketitle

\section{Introduction}
The question of bi-Lipschitz classification of semialgebraic surfaces has become in recent years one of the central questions of Metric Geometry of Singularities. There are two natural structures of a metric space on a connected semialgebraic set $X\subset\R^n$. The first one is the inner distance, the length of a minimal path in $X$ connecting two points. The second one is the outer distance, defined as the distance in $\R^n$ between two points of $X$.
A germ $X$ is called \emph{normally embedded} (see \cite{birbrair2000normal}) if its inner and outer metrics are equivalent.
There are three natural equivalence relations associated with these distances. Two sets $X$ and $Y$ are inner (resp., outer) equivalent if there is a inner (resp., outer) bi-Lipschitz homeomorphism $h:X\to Y$. The sets $X$ and $Y$ are ambient bi-Lipschitz equivalent if the homeomorphism $h:X\to Y$ can be extended to a bi-Lipschitz homeomorphism $H$ of the ambient space. The ambient equivalence is stronger than the outer equivalence, and the outer equivalence is stronger then the inner equivalence.  Finiteness theorems of Mostowski and Valette (see \cite{Mostowski} and  \cite{valette2005Lip}) show that there are finitely many ambient bi-Lipschitz equivalence classes in any semialgebraic family of semialgebraic sets.

The paper \cite{birbrair1999} of the first author presents a complete bi-Lipschitz classification of semialgebraic surface germs at the origin with respect to the inner metric.
It is based on a canonical partition of a surface germ into H\"older triangles and isolated arcs.
The exponents of these triangles, and the combinatorics of the graph defined by their links, constitute a complete inner Lipschitz invariant.

The outer Lipschitz geometry of semialgebraic surface germs is considerably more complicated, and their outer bi-Lipschitz classification is still work in progress. An important step towards such classification was made in \cite{birbrair2014lipschitz}, where classification of the germs at the origin of $\R^2$ of semialgebraic (or, more generally, definable in a polynomially bounded o-minimal structure) Lipschitz functions with respect to contact Lipschitz equivalence relation was suggested.
It was based on a complete combinatorial invariant of contact Lipschitz equivalence, called \emph{pizza}.

Another important step was made in \cite{GS}, where an ``abnormal'' semialgebraic surface germ was canonically partitioned into normally embedded H\"older triangles. Several constructions and results from \cite{GS} are used in the present paper.

Normally embedded H\"older triangles are the simplest ``building blocks'' of semialgebraic surface germs: the only Lipschitz invariant of a normally embedded H\"older triangle is its exponent.
In the present paper we consider the next, a little bit more complicated, case of a \emph{pair} of normally embedded H\"older triangles: a surface germ $X=T\cup T'$ which is the union of two normally embedded H\"older triangles $T$ and $T'$.
Let $f:T\to\R$ and $g:T'\to\R$ be the distances from the points in one of these two triangles to the other one.
The pizzas of $f$ and $g$, being contact Lipschitz invariants of these two functions, are outer Lipschitz invariants of $X$.
The first question is whether $X$ is outer bi-Lipschitz equivalent to the union of $T$ and the graph of the distance function $f$.
Simple examples (see Fig.~\ref{fig:non-elementary}) show that the answer may be negative.
Another natural question is whether the pizzas of $f$ and $g$ are equivalent. The answer, in general, is again negative.
We show (see Theorem \ref{theorem:very-elementary}) that the answers to both questions are positive if the pair $(T,T')$ is \emph{elementary} (see Definition \ref{def: Elementary Holder triangle}) and satisfies boundary conditions (\ref{tord-tord}).
The conditions (\ref{tord-tord}) appear naturally in the paper \cite{GS}, where some standard building blocks (\emph{clusters}) are defined in the link of a singular surface. Any two H\"older triangles in a cluster satisfy (\ref{tord-tord}).
Although a pair $X$ satisfying (\ref{tord-tord}) is simpler than the general pair of normally embedded H\"older triangles, its outer Lipschitz geometry is still rather complicated.
If one considers a pair $X=T\cup T'$ of two normally embedded H\"older triangles, such that $T'$ is a graph of a Lipschitz function defined on $T$, then $X$ automatically satisfies the condition (\ref{tord-tord}). A natural question is whether the opposite is true. Suppose that a pair
$X=T\cup T'$ of normally embedded H\"older triangles satisfies (\ref{tord-tord}). Is it true that $X$ is outer Lipschitz equivalent to the union of $T$ and the graph of a function $f$ defined on $T$? The answer is negative, and we present several examples when this is not true (see Section \ref{the-invariant}).
In this paper we define an outer Lipschitz invariant of a pair of normally embedded H\"older triangles satisfying (\ref{tord-tord}), called \emph{$\sigma\tau$-pizza}, and conjecture that it is a complete invariant: all
pairs with the same $\sigma\tau$-pizza should be outer bi-Lipschitz equivalent.

In Section \ref{prelim} we give basic definitions and reformulate the pizza invariant in the language of \emph{zones} (see Definition \ref{zone}).

In Section \ref{Section:LNE triangles} we establish properties of elementary pairs of H\"older triangles and give examples
of non-elementary pairs for which these properties fail.
We also discuss conditions satisfied by a surface germ $X=T\cup T'$ equivalent to the union of a H\"older triangle $T$ and the graph of the distance function $f$ defined on $T$.

In Section \ref{the-invariant} the $\sigma\tau$-pizza is defined.
The main result of the section, Theorem \ref{invariant}, states that it is an outer Lipschitz invariant of a pair of normally embedded H\"older triangles satisfying (\ref{tord-tord}): the $\sigma\tau$-pizzas of outer bi-Lipschitz equivalent pairs are combinatorially equivalent.
We conjecture that the converse of Theorem \ref{invariant} is also true, but the proof needs some additional work.

Some remarks about the figures. Since it is practically impossible to adequately show outer Lipschitz geometry of a surface germ in a plot, we draw instead its link (intersection with a small sphere centered at the singular point) indicating higher tangency orders by smaller Euclidean distances. We hope these plots will help to create geometric intuition.

\section{Preliminaries}\label{prelim}

All sets, functions and maps in this paper are germs at the origin of $\R^n$ definable in a polynomially bounded o-minimal structure over $\R$ with the field of exponents $\F$.
The simplest (and most important in applications) examples of such structures are real semialgebraic and (global) subanalytic sets, with ${\F}={\Q}$.

\begin{definition}\label{metrics}\normalfont
A germ $X$ at the origin inherits two metrics from the ambient space: the \emph{inner metric} where the distance between two points of
$X$ is the length of the shortest path connecting them inside $X$, and the \emph{outer metric} with the distance between two points of
$X$ being their distance in the ambient space.
A germ $X$ is \emph{normally embedded} if its inner and outer metrics are equivalent.

For a point $x\in X$ and a subset $Y\subset X$ we define the \emph{outer distance} $dist(x,Y)=\inf_{y\in Y}|x-y|$, and the \emph{inner distance} $idist(x,Y)$ as the infimum of the lengths of paths connecting $x$ with points $y\in Y$.

A \emph{surface germ} is a closed germ $X$ such that $\dim_\R X=2$, and it is pure dimensional.
\end{definition}

\begin{definition}\label{arc}\normalfont
An \emph{arc} in $\R^n$ is (a germ at the origin of) a mapping $\gamma:[0,\epsilon)\rightarrow \R^n$ such that $\gamma(0)=0$.
Unless otherwise specified, we suppose that arcs are parameterized by the distance to the origin, i.e., $|\gamma(t)|=t$.
We usually identify an arc $\gamma$ with its image in $\R^n$.
The \emph{Valette link} of $X$ is the set $V(X)$ of all arcs $\gamma\subset X$.
\end{definition}

\begin{definition}\label{ordonarc}
\emph{Let $f\not\equiv 0$ be (a germ at the origin of) a function defined on an arc $\gamma$.
The \emph{order} of $f$ on $\gamma$ is the value $q=ord_\gamma f\in\F$ such that $f(\gamma(t))=c t^q+o(t^q)$ as $t\to 0$, where $c\ne 0$.
If $f\equiv 0$ on $\gamma$, we set $ord_\gamma f=\infty$.}
\end{definition}

\begin{definition}\label{tord}\normalfont
The \emph{tangency order} of two arcs $\gamma$ and $\gamma'$ is defined as  $tord(\gamma,\gamma')=ord_{\gamma}
|\gamma(t)-\gamma'(t)|$.
The tangency order of an arc $\gamma$ and a set of arcs $Z\subset V(X)$  is defined as
$tord(\gamma,Z)=\sup_{\lambda\in Z} tord(\gamma,\lambda)$.
The tangency order of two subsets $Z$ and $Z'$ of $V(X)$ is defined as $tord(Z,Z')=\sup_{\gamma\in Z} tord(\gamma,Z')$.
Similarly, $itord(\gamma,\gamma')$, $itord(\gamma,Z)$ and $itord(Z,Z')$ denote the tangency orders with respect to the inner metric.
If $T$ is a H\"older triangle and $\gamma$ is an arc we are going to use the notation $tord(\gamma, T)$ instead of $tord(\gamma, V(T))$ and $itord(\gamma, T)$ instead of $itord(\gamma, V(T))$.

The tangency order defines a non-Archimedean metric on the set of arcs: if $tord(\gamma,\gamma')>tord(\gamma,\gamma'')$ then
$tord(\gamma',\gamma'')=tord(\gamma,\gamma'')$.
\end{definition}

\begin{remark} \normalfont
 The inner metric on a semialgebraic set is bi-Lipschitz equivalent to a semialgebraic metric (so-called pancake metric, see the theorem of Kurdyka and Orro \cite{kur-orr} and also \cite{birbrair2000normal}). The inner order of tangency of two arks $itord(\gamma_1,\gamma_2))$ is also defined in \cite{birbrair-mendes}.
\end{remark}
\begin{definition}\label{standard holder}
	\normalfont For $\beta \in \F$, $\beta \ge 1$, the \emph{standard $\beta$-H\"older triangle} is (a germ at the origin of) the set
	\begin{equation}\label{Formula:Standard Holder triangle}
	T_\beta = \{(x,y)\in \R^2 \mid x\ge 0, \; 0\le y \le x^\beta\}.
	\end{equation}
	The curves $\{x\ge 0,\; y=0\}$ and $\{x\ge 0,\; y=x^\beta\}$ are the \emph{boundary arcs} of $T_\beta$.
\end{definition}

\begin{definition}\label{holder}\normalfont
A $\beta$-\emph{H\"older triangle} is (a germ at the origin of) a set $T \subset \R^n$ that is inner bi-Lipschitz homeomorphic
to the standard $\beta$-H\"older triangle (\ref{Formula:Standard Holder triangle}).
The number $\beta=\mu(T) \in \F$ is called the \emph{exponent} of $T$.
The arcs $\gamma_1$ and $\gamma_2$ of $T$ mapped to the boundary arcs of $T_\beta$ by an inner bi-Lipschitz homeomorphism are the
\emph{boundary arcs} of $T$ (notation $T=T(\gamma_1,\gamma_2)$).
All other arcs of $T$ are its \emph{interior arcs}. The set of interior arcs of $T$ is denoted by $I(T)$.
An arc $\gamma\subset T$ is \emph{generic} if $itord(\gamma,\gamma_1)=itord(\gamma,\gamma_2)$.
The set of generic arcs of $T$ is denoted by $G(T)$.
\end{definition}

\begin{definition}\label{singulararc}
\normalfont Let $X$ be a surface germ.
An arc $\gamma\in V(X)$ is \emph{Lipschitz non-singular} if there exists a normally embedded H\"older triangle
$T\subset X$ such that $\gamma\in I(T)$ and $\gamma\not\subset\overline{X\setminus T}$.
Otherwise, $\gamma$ is \emph{Lipschitz singular}.
A H\"older triangle $T$ is \emph{non-singular} if any arc $\gamma\in I(T)$ is Lipschitz non-singular.
\end{definition}

\begin{definition}\label{def:Qf}
	\emph{For a Lipschitz function $f$ defined on a H\"older triangle $T$, let
	\begin{equation}\label{eqn:Qf}
	Q_{f}(T)=\bigcup_{\gamma\in V(T)}ord_{\gamma} f.
	\end{equation}
	It was shown in \cite{birbrair2014lipschitz} that $Q_f(T)$ is either a point or a closed interval in $\mathbb{F}\cup \{\infty\}$.}
\end{definition}

\begin{definition}\label{def: Elementary Holder triangle}
	\emph{A H\"older triangle $T$ is \emph{elementary} with respect to a Lipschitz function $f$ if, for any $q\in Q_f(T)$ and any two arcs
$\gamma$ and $\gamma'$ in $T$ such that $ord_{\gamma}f=ord_{\gamma'}f=q$, the order of $f$ is $q$ on any arc in the H\"older triangle
$T(\gamma,\gamma')\subset T$.}
\end{definition}

\begin{remark}\normalfont
Examples 4.4, 4.5, 4.6 in \cite{birbrair2014lipschitz} make the definition \ref{def: Elementary Holder triangle} more clear.
\end{remark}

\begin{definition}\label{def:width function}\normalfont Let $T$ be a H\"older triangle and $f$ a Lipschitz function defined on $T$.
For each arc $\gamma \subset T$, the \emph{width} $\mu_T(\gamma,f)$ of $\gamma$ with respect to $f$ is the infimum of exponents of H\"older triangles $T'\subset T$ containing $\gamma$ such that $Q_f(T')$ is a point.
For $q\in Q_f(T)$ let $\mu_{T,f}(q)$ be the set of exponents $\mu_T(\gamma,f)$, where $\gamma$ is any arc in $T$ such that $ord_{\gamma}f=q$.
It was shown in \cite{birbrair2014lipschitz} that, for each $q\in Q_f(T)$, the set $\mu_{T,f}(q)$ is finite.
This defines a multivalued \emph{width function} $\mu_{T,f}: Q_f(T)\to \F\cup \{\infty\}$.
If $T$ is an elementary H\"older triangle with respect to $f$ then the function $\mu_{T,f}$ is single valued.
When $f$ is fixed, we write $\mu_T(\gamma)$ and $\mu_T$ instead of $\mu_T(\gamma,f)$ and $\mu_{T,f}$.

The \emph{depth} $\nu_T(\gamma,f)$ of an arc $\gamma$ with respect to $f$
is the infimum of exponents of H\"older triangles $T'\subset T$ such that $\gamma\in G(T')$ and $Q_f(T')$ is a point.
By definition, $\nu_T(\gamma,f)=\infty$ when there are no such triangles $T'$.
\end{definition}

\begin{definition}\label{def: pizza slice}
	\emph{Let $T$ be a non-singular H\"older triangle and $f$ a Lipschitz function defined on $T$.
We say that $T$ is a \emph{pizza slice} associated with
$f$ if it is elementary with respect to $f$ and, unless $Q_f(T)$ is a point, $\mu_{T,f}(q)=aq+b$ is an affine function on $Q_f(T)$.
If $T$ is a pizza slice such that $Q_f(T)$ is not a point, then the \emph{supporting arc}
$\tilde\gamma$ of $T$ with respect to $f$ is the boundary arc of $T$ such that $\mu_T(\tilde\gamma,f)=\max_{q\in Q_f(T)}\mu_{T,f}(q)$.}
\end{definition}

\begin{proposition}\label{prop:width function properties}
	\emph{(See \cite{birbrair2014lipschitz}.)} Let $T$ be a $\beta$-H\"older triangle which is a pizza slice associated with a non-negative Lipschitz function $f$, such that $Q=Q_f(T)$ is not a point. Then $\mu_T\not\equiv \mathrm{const}$ and the following holds:\newline
\emph{(1)} $\beta\le\mu_T(q)\le \max(q,\beta)$ for all $q \in Q$,\newline
\emph{(2)} $\mu_T(\gamma)=\beta$ for $\gamma \in G(T)$,\newline
\emph{(3)} If $\tilde\gamma$ is the supporting arc of $T$ with respect to $f$,
		then $\mu_T(\gamma)=itord(\tilde\gamma,\gamma)$ for all arcs $\gamma\subset T$ such that $\mu_T(\gamma)<\mu_T(\tilde\gamma)$.
\end{proposition}

\begin{definition}\label{pizza-decomp}\emph{(See \cite{birbrair2014lipschitz}.)}
		\emph{Let $f$ be a non-negative Lipschitz function defined on an oriented  $\beta$-H\"older triangle $T$.
A \emph{pizza decomposition} of $T$ (or just a \emph{pizza} on $T$) associated with $f$ is a decomposition $\{T_i\}_{i=1}^p$ of $T$ into
$\beta_i$-H\"older triangles $T_i=T(\lambda_{i-1},\lambda_i)$ ordered according to the orientation of $T$, such that\newline
(1) $\lambda_0$ and $\lambda_p$ are the boundary arcs of $T$,\newline
(2) $T_i\cap T_{i+1}=\lambda_i$ for $1\le i< p$,\newline
(3) $T_i\cap T_j=\{0\}$ when $|i-j|>1$,\newline
(4) each H\"older triangle $T_i$ is a pizza slice associated with $f$.\newline
We write $q_i=ord_{\lambda_i}f$, $Q_i=Q_f (T_i)$, $\mu_i(q)=\mu_{T_i,f}(q)$.
If $Q_i$ is not a point, then $\tilde\gamma_i$ denotes the supporting arc of $T_i$ with respect to $f$.}
\end{definition}

\begin{definition}\label{pizza-minimal}	\emph{A pizza decomposition $\{T_i\}$ of $T$ associated with $f$ is \emph{minimal} if $T_{i-1}\cup T_i$ is not a pizza slice associated with $f$ for any $i>1$.}
\end{definition}

\begin{definition}\label{pizza-equiv}		
\emph{For two non-negative Lipschitz functions $f$ on $T$ and $g$ on $T'$, a pizza decomposition $\{T_i=T(\lambda_{i-1},\lambda_i)\}$ of $T$ associated with $f$
is \emph{equivalent} to a pizza decomposition $\{T'_i=T(\lambda'_{i-1},\lambda'_i)\}$ of $T'$ associated with $g$
if there is an orientation preserving inner bi-Lipschitz homeomorphism $h:T\to T'$ such that $h(\lambda_i)=\lambda'_i$, $ord_{\lambda_i}f=ord_{\lambda'_i}g$, $Q_f(T_i)=Q_g (T'_i)$ and $\mu_{T_i,f}\equiv\mu_{T'_i,g}$, for all $i$, and moreover, $h(\tilde\gamma_i)=\tilde\gamma'_i$ if $Q_f(T_i)=Q_g (T'_i)$ is not a point, where $\tilde\gamma_i$ and $\tilde\gamma'_i$ are the supporting arcs for $T_i$ and $T'_i$ with respect to $f$ and $g$.}
\end{definition}

\begin{definition}\label{def:contactequiv}
\emph{ Let $T$ and $T'$ be two $\beta$-H\"older triangles.
Two Lipschitz function germs $f: (T,0) \longrightarrow (\R,0)$ and  $g: (T',0) \longrightarrow (\R,0)$ are
\emph{Lipschitz contact equivalent} if there exist two germs of inner
bi-Lipschitz homeomorphisms $h:(T,0) \longrightarrow (T',0)$ and  $H: (T\times\R,0) \longrightarrow (T'\times\R,0)$ such that
$H(T\times \{0\}) = T'\times \{0\}$ and the following
diagram is commutative:}

\begin{equation}\label{diagram:contactequiv}
\begin{array}{lllll}
(T,0) & \stackrel{(id,\, f)}{\longrightarrow}
 & (T \times \R,0) & \stackrel{\pi}{\longrightarrow} & (T,0) \\
\,\,\,h \, \downarrow &  & \;\;\;\;\; H \, \downarrow  &  & \,\,\, h \, \downarrow \\
(T',0)& \stackrel{(id, \,g)}{\longrightarrow} & (T' \times\R,0) & \stackrel{\pi'}{\longrightarrow}& (T',0) \\
\end{array}
\end{equation}
\emph{Here $\pi:T \times \R\to T$ and $\pi': T'\times \R\to T'$ are natural projections.}
\end{definition}

The main result of \cite{birbrair2014lipschitz}, reformulated for non-negative Lipschitz functions defined on H\"older triangles, is the following theorem.

\begin{theorem}\label{pizza-theorem} Let $T$ and $T'$ be oriented H\"older triangles. Non-negative Lipschitz functions $f:T\to \R$ and $g:T'\to \R$ are Lipschitz contact equivalent if and only if a minimal pizza decomposition of $T$ associated with $f$ and a minimal pizza decomposition of $T'$ associated with $g$ are equivalent.
In particular, any two minimal pizza decompositions associated with the same function $f:T\to \R$ are equivalent.
\end{theorem}

\begin{definition}\label{zone}\normalfont (See \cite[Definition 2.34]{GS}.) Let $X$ be a surface germ.
A non-empty set of arcs $Z\subset V(X)$ is called a \emph{zone} if, for any two arcs $\gamma_1\ne\gamma_2$ in $Z$, there exists a
non-singular H\"older triangle $T=T(\gamma_1,\gamma_2)$ such that $V(T)\subset Z$.
A \emph{singular zone} is a zone $Z=\{\gamma\}$ consisting of a single arc $\gamma$.
A zone $Z$ is \emph{normally embedded} if, for any two arcs $\gamma_1\ne\gamma_2$ in $Z$,
there exists a normally embedded H\"older triangle $T=T(\gamma_1,\gamma_2)$ such that $V(T)\subset Z$.
\end{definition}

\begin{definition}\label{order}
\emph{(See \cite[Definition 2.37]{GS}.)
The \emph{order} of a zone $Z$ is defined as $\mu(Z)=\inf_{\gamma,\gamma'\in Z} tord(\gamma,\gamma')$.
If $Z$ is a singular zone then $\mu(Z)=\infty$. If $\mu(Z)=\beta$ then $Z$ is called a $\beta$-zone.}
\end{definition}

\begin{definition}\label{closed}
\emph{(See \cite[Definition 2.40]{GS}.)
A $\beta$-zone $Z$ is \emph{closed} if there are two arcs $\gamma$ and $\gamma'$ in $Z$ such that $tord(\gamma,\gamma')=\beta$.
Otherwise, $Z$ is an \emph{open} zone. By definition, any singular zone is closed.}
\end{definition}

\begin{definition}\label{perfect}
\emph{A zone $Z\subset V(X)$ is \emph{perfect} if, for any two arcs $\gamma$ and $\gamma'$ in $Z$, there exists a H\"older triangle $T\subset X$ such that
$V(T)\subset Z$ and both $\gamma$ and $\gamma'$ are generic arcs of $T$. By definition, any singular zone is perfect.}
\end{definition}

\begin{definition}\label{q-zone}\normalfont
Let $f:T\to\R$ be a Lipschitz function defined on a non-singular H\"older triangle $T$.
A zone $Z\subset V(T)$ is a $q$-\emph{order zone} for $f$ if $ord_\gamma f=q$ for any arc $\gamma\in Z$.
A $q$-order zone for $f$ is \emph{maximal} if it is not a proper subset of any other $q$-order zone for $f$.
The \emph{width zone} $W_T(\gamma,f)$ of an arc $\gamma\subset T$ with respect to $f$ is the maximal $q$-order zone
for $f$ containing $\gamma$, where $q=ord_\gamma f$.
The order of $W_T(\gamma,f)$ is $\mu_T(\gamma,f)$.
The \emph{depth zone} $D_T(\gamma,f)$ of an arc $\gamma\subset T$ with respect to $f$ is
the union of zones $G(T')$ for all triangles $T'\subset T$ such that $\gamma\in G(T')$ and $Q_f(T')$ is a point.
By definition, $D_T(\gamma,f)=\{\gamma\}$ when there are no such triangles $T'$.
The order of $D_T(\gamma,f)$ is $\nu_T(\gamma,f)$.
\end{definition}

\begin{lemma}\label{width zone is closed}
Let $f:T\to\R$ be a Lipschitz function defined on a non-singular H\"older triangle $T$.
For any arc $\gamma\subset T$, the width zone $W_T(\gamma,f)$ is closed.
\end{lemma}

\begin{proof}
If $f|_\gamma\equiv 0$, then either $\gamma$ is an isolated arc in the closed subset $T_0=\{f(x)=0\}$ of $T$
 and a singular zone $W_T(\gamma,f)=\{\gamma\}$ is closed by definition,
or there is a maximal H\"older triangle $\tilde T_0\subset T_0$ containing $\gamma$.
Then $\mu=\mu_T(\gamma,f)$ is the exponent of $\tilde T_0$, and $W_T(\gamma,f)=V(\tilde T_0)$ is a closed $\mu$-zone.
Otherwise, let $f(\gamma(t))=c_0 t^q+o(t^q)$ where $c_0\ne 0$, and let $\tilde T_c$ be the maximal H\"older triangle
containing $\gamma$ in the subset $T_c=\{|f(x)|\le c t^q\}$ of $T$, where $c\ge|c_0|$.
Then the family $\{\tilde T_c\}$ is definable, H\"older triangles $\tilde T_c$ have the same exponent $\mu=\mu_T(\gamma,f)$
for large enough $c$, and $W_T(\gamma,f)=\bigcup_{c\ge|c_0|}V(\tilde T_c)$.
Thus $W_T(\gamma,f)$ is a closed $\mu$-zone.
\end{proof}

\begin{definition}\label{def:zone-slice}
\emph{Let $T$ be a non-singular H\"older triangle and $f$ a Lipschitz function defined on $T$. If $Z\subset V(T)$ is a zone, we define $Q_f(Z)$
as the set of all exponents $ord_\gamma f$ for $\gamma\in Z$. The zone $Z$ is \emph{elementary} with respect to $f$ if the set of arcs $\gamma\in Z$
such that $ord_\gamma f=q$ is a zone for each $q\in Q_f(Z)$.}

\emph{For $\gamma\in Z$ and $q=ord_\gamma f$, the \emph{width} $\mu_Z(\gamma,f)$ of $\gamma$ with respect to $f$ is the infimum of exponents of H\"older triangles $T'$ containing $\gamma$ such that $V(T')\subset Z$ and $Q_f(T')$ is a point.
The \emph{width zone} $W_Z(\gamma,f)$ of $\gamma$ with respect to $f$ is the maximal subzone of $Z$ containing $\gamma$ such that $q=ord_\lambda f$ for all arcs $\lambda\subset W_Z(\gamma,f)$. The order of $W_Z(\gamma,f)$ is $\mu_Z(\gamma,f)$.
For $q\in Q_f(Z)$ let $\mu_{Z,f}(q)$ be the set of exponents $\mu_Z(\gamma,f)$, where $\gamma\in Z$ is any arc such that $ord_{\gamma}f=q$.
It follows from \cite{birbrair2014lipschitz} that, for each $q\in Q_f(Z)$, the set $\mu_{Z,f}(q)$ is finite.
This defines a multivalued \emph{width function} $\mu_{Z,f}: Q_f(Z)\to \F\cup \{\infty\}$.
If $Z$ is an elementary zone with respect to $f$ then the function $\mu_{Z,f}$ is single valued.}

\emph{We say that $Z$ is a \emph{pizza slice zone} associated with
$f$ if it is elementary with respect to $f$, $Q_f(Z)$ is a closed interval in $\F\cup\{\infty\}$ and, unless $Q_f(Z)$ is a point, $\mu_{Z,f}(q)=aq+b$ is an affine function on $Q_f(Z)$.
If $Z$ is a pizza slice zone such that $Q_f(Z)$ is not a point, then the \emph{supporting subzone}
$\tilde Z$ of $Z$ with respect to $f$ is the set of arcs $\lambda\in Z$ such that $\mu_Z(\lambda,f)=\max_{q\in Q_f(Z)}\mu_{Z,f}(q)$.}
\end{definition}

\begin{lemma}\label{depth}
Let $f$ be a Lipschitz function defined on a non-singular H\"older triangle $T$.
Let $\gamma$ be an interior arc of $T$, so that $T=T'\cup T''$ and $T'\cap T''=\{\gamma\}$.
Then either $\mu_{T'}(\gamma,f)=\mu_{T''}(\gamma,f)$ and $\nu_T(\gamma,f)=\mu_T(\gamma,f)$, or $\nu_T(\gamma,f)=\max(\mu_{T'}(\gamma,f),\mu_{T''}(\gamma,f))>\mu_T(\gamma,f)$.
In both cases, $D_T(\gamma,f)$ is a closed perfect zone.
\end{lemma}

\begin{proof} Let $\mu=\mu_T(\gamma,f)$, $\mu'=\mu_{T'}(\gamma,f)$ and $\mu''=\mu_{T''}(\gamma,f)$.
By definition of the width, $\mu=\min(\mu',\mu'')$. By definition of the depth, $\nu_T(\gamma,f)\ge\max(\mu',\mu'')$.
According to Lemma \ref{width zone is closed}, the width zones $W_{T'}(\gamma,f)$ and $W_{T''}(\gamma,f)$ are closed zones of orders $\mu'$ and $\mu''$.
If $\mu'=\mu''=\mu$ then there are two arcs $\gamma'\subset W_{T'}(\gamma,f)$ and $\gamma''\subset W_{T''}(\gamma,f)$
such that $tord(\gamma,\gamma')=tord(\gamma,\gamma'')=\mu$ and $ord_{\lambda}f=ord_\gamma f$ for all arcs $\lambda\subset T(\gamma',\gamma'')$.
Then $\gamma$ is a generic arc of a $\mu$-H\"older triangle $T(\gamma',\gamma'')$, thus $\nu_T(\gamma,f)\le\mu$.
Since $\nu_T(\gamma,f)\ge\mu$, we have $\nu_T(\gamma,f)=\mu$ in this case.
Otherwise, if $\mu'>\mu''$ then, according to Lemma \ref{width zone is closed}, there are two arcs $\gamma'\subset W_{T'}(\gamma,f)$ and $\gamma''\subset W_{T''}(\gamma,f)$ such that $tord(\gamma,\gamma')=tord(\gamma,\gamma'')=\mu'$ and $ord_{\lambda}f=ord_\gamma f$ for all arcs $\lambda\subset T(\gamma',\gamma'')$.
Then $\gamma$ is a generic arc of a $\mu'$-H\"older triangle $T(\gamma',\gamma'')$, thus $\nu_T(\gamma,f)\le\mu'$.
Since $\nu_T(\gamma,f)\ge\max(\mu',\mu'')$, we have $\nu_T(\gamma,f)=\max(\mu',\mu'')$ in this case.

To show that $D_T(\gamma,f)$ is a closed perfect zone, note first that its order is $\nu=\nu_T(\gamma,f)$ and, unless $\nu=\infty$
and $D_T(\gamma,f)=\{\gamma\}$ is by definition closed perfect, $\gamma$ is a generic arc
of a $\nu$-H\"older triangle $\tilde T=T(\gamma',\gamma'')\subset T$ such that $tord_\lambda f=tord_\gamma f$ for any arc $\lambda\subset\tilde T$.
Then, since $tord(\gamma,\gamma')=tord(\gamma,\gamma'')=\nu<\infty$, there is a generic arc $\lambda$ of $\tilde T$ such that $tord(\lambda,\gamma)=\nu$,
thus $\bar T=T(\lambda,\gamma)$ is a $\nu$-H\"older triangle and $V(\bar T)\subset D_T(\gamma,f)$. This implies that $D_T(\gamma,f)$
is a closed zone. If $\lambda'$ and $\lambda''$ are any two arcs in $D_T(\gamma,f)$, then there are two H\"older triangles
$T'\subset T$ and $T''\subset T$ containing $\gamma$ such that $\lambda'\in G(T')$ and $\lambda''\in G(T'')$.
Then both $\lambda'$ and $\lambda''$ are generic arcs of $T'\cup T''$, thus $D_T(\gamma,f)$ is a perfect zone.
\end{proof}

\begin{remark}\label{width-depth}\emph{Let $h:T\to T'$ be an inner bi-Lipschitz homeomorphism, and let $f(x)=g(h(x))$ where $g$ is a Lipschitz
function defined on $T'$. Then $\mu_T(\gamma,f)=\mu_{T'}(h(\gamma),g)$, $\nu_T(\gamma,f)=\nu_{T'}(h(\gamma),g)$,
$h(W_T(\gamma,f))=W_{T'}(h(\gamma),g)$ and $h(D_T(\gamma,f))=D_{T'}(h(\gamma),g)$, for any arc $\gamma\in V(T)$.}
\end{remark}

\begin{lemma}\label{perfect-automorphism}
 A zone $Z\subset V(X)$ is perfect if and only if, for any two arcs $\gamma$ and $\gamma'$ in $Z$, there exists a H\"older triangle $T\subset X$ such that $V(T)\subset Z$, and an inner bi-Lipschitz automorphism $h:X\to X$ such that $h(\gamma)=\gamma'$ and $h(x)=x$ for all $x\in X\setminus T$.
\end{lemma}

\begin{proof}
Let $Z\subset V(X)$ be a perfect zone and $\gamma,\gamma'$ two arcs in $Z$. Let $T=T(\gamma_1,\gamma_2)$ be a $\beta$-H\"older triangle such that  $V(T)\subset Z$ and both $\gamma$ and $\gamma'$ are generic arcs in $T$. Then $T=T(\gamma_1,\gamma)\cup T(\gamma,\gamma_2)$ and $T=T(\gamma_1,\gamma')\cup T(\gamma',\gamma_2)$ are two decompositions of $T$ into $\beta$-H\"older triangles. Let $h_1: T(\gamma_1,\gamma)\to T(\gamma_1,\gamma')$ and $h_2: T(\gamma,\gamma_2)\to T(\gamma',\gamma_2)$ be inner bi-Lipschitz homeomorphisms, such that $h_1 |_{\gamma_1=Id}$, $h_2 |_{\gamma_2}=Id$ and
$h_1 |_\gamma=h_2 |_\gamma$. Then the mapping $h:T\to T$ such that $h=h_1$ on $T(\gamma_1,\gamma)$ and $h=h_2$ on $T(\gamma,\gamma_2)$
is an inner bi-Lipschitz homeomorphism such that  $h |_{\gamma_1}=Id$, $h |_{\gamma_2}=Id$ and $h(\gamma)=\gamma')$.
Thus $h$ can be extended  by identity outside $T$ to an inner bi-Lipschitz homeomorphism $X\to X$ preserving $Z$.
\end{proof}

\begin{proposition}\label{MP}
Let $f$ be a non-negative Lipschitz function defined on a normally embedded H\"older triangle $T=T(\gamma_1,\gamma_2)$, oriented from $\gamma_1$ to $\gamma_2$.
There exists a unique finite family $\{D_\ell\}_{\ell=0}^p$ of disjoint zones $D_\ell\subset V(T)$, the pizza zones associated with $f$, with the following properties:\newline
\emph{1.} The singular zones $D_0=\{\gamma_1\}$ and $D_p=\{\gamma_2\}$ are the boundary arcs of $T$.\newline
\emph{2.} For any arc $\gamma\in D_\ell$, $D_\ell=D_T(\gamma,f)$ is a closed perfect $\nu_\ell$-zone, where $\nu_\ell=\nu_T(\gamma,f)$.
In particular, $D_\ell$ is a $q_\ell$-order zone for $f$, where $q_\ell=ord_\gamma f$ for $\gamma\in D_\ell$.
Moreover, $D_\ell$ is a maximal $q_\ell$-order zone for $f$ of order $\nu_\ell$: if $Z\subset V(T)$ is a $q_\ell$-order zone for $f$ containing $D_\ell$ and $\lambda\in Z$ is an arc such that $tord(\lambda,D_\ell)\ge\nu_\ell$, then $\lambda\in D_\ell$.\newline
\emph{3.} Any choice of arcs $\lambda_\ell\in D_\ell$ defines a minimal pizza $\{T_\ell=T(\lambda_{\ell-1},\lambda_\ell)\}_{\ell=1}^p$ on $T$ associated with $f$.\newline
\emph{4.} Any minimal pizza on $T$ associated with $f$ can be obtained as a decomposition $\{T_\ell\}$ of $T$ defined by some choice of arcs $\lambda_\ell\in D_\ell$.
\end{proposition}

\begin{proof}
Consider a decomposition $\{T_\ell\}_{\ell=1}^p$, of $T$ into $\beta_\ell$-H\"older triangles $T_\ell=T(\lambda_{\ell-1},\lambda_\ell)$ which is a minimal pizza for $f$.
Let $Q_\ell\subset\F\cup\{\infty\}$ be the set (either a point or a closed interval) of values
$tord_\gamma f$ for $\gamma\subset T_\ell$, and let $\mu_\ell:Q_\ell\to\F\cup\{\infty\}$ be the affine width function for $f$ on $T_\ell$
(a constant if $Q_\ell$ is a point).
We assume that $\lambda_0$ and $\lambda_p$ are the boundary arcs of $T$, and that $T_\ell\cap T_{\ell+1}=\lambda_\ell$ for $1\le \ell<p$.

Since each boundary arc of $T$ is also a boundary arc of a pizza slice for any pizza decomposition of $T$,
we can define singular zones $D_0=\{\lambda_0\}$ and $D_p=\{\lambda_p\}$.

If the germ at zero of the set $S=\{x\in T,\;f(x)=0\}$ is non-empty, it is a union of finitely many germs isolated arcs and germs of maximal in $S$ H\"older triangles $S_j$.
Each isolated arc of $S$, and each boundary arc of one of the triangles $S_j$, must be a boundary arc of a pizza slice for any minimal pizza on $T$ associated with $f$.
In particular, such an arc $\lambda$ must be one of the arcs $\lambda_\ell$, and the singular zone $\{\lambda\}$ must be one of the zones $D_\ell$.

Assume now that $0<\ell<p$ and $q_\ell=ord_{\lambda_\ell} f<\infty$.
Consider the depth zone $D_\ell=D_T(\lambda_\ell,f)$ (see Definition \ref{q-zone}).
Then $D_\ell$ is a closed perfect zone of order $\nu_\ell=\nu_T(\lambda_\ell,f)$,
which is also a $q_\ell$-order zone for $f$.
Moreover, if $\lambda\subset T_\ell$ is an arc such that $tord(\lambda,\lambda_\ell)\ge\nu_\ell$ and $ord_\gamma f=q_\ell$ for any arc $\gamma\subset T(\lambda_\ell,\lambda)$, then $\lambda\in D_\ell$ by Definition \ref{q-zone}.
The same argument works for $\lambda\subset T_{\ell-1}$.
Thus $D_\ell$ is a maximal $q_\ell$-order zone for $f$ of order $\nu_\ell$.

We claim that, if the arc $\lambda_\ell$ is replaced by any other arc $\theta\in D_\ell$ and the H\"older triangles
$T_\ell=T(\lambda_{\ell-1},\lambda_\ell)$ and $T_{\ell+1}=T(\lambda_\ell,\lambda_{\ell+1})$ with the common arc $\lambda_\ell$ are
replaced by the H\"older triangles $T(\lambda_{\ell-1},\theta)$ and $T(\theta,\lambda_{\ell+1})$ with the common arc $\theta$,
the resulting decomposition of $T$ is again a minimal pizza on $T$ associated with $f$.
Indeed, since $D_\ell$ is a perfect zone, and also a $q_\ell$-order zone for $f$,
by Lemma \ref{perfect-automorphism} one can construct  an inner bi-Lipschitz map $\phi:T\to T$, such that
$\phi(\lambda_\ell)=\theta$ and $\phi(\gamma)=\gamma$ for any arc $\gamma\in V(T)\setminus D_\ell$.
In particular, $ord_{\phi(\gamma)} f=ord_\gamma f$ for each arc $\gamma\subset T$, thus
$\phi$ transforms the function $f$ into a $v$-equivalent function.
This implies that $\phi$ preserves all zones $D_\ell$, and that decomposition $\{\phi(T_\ell)\}$ defines a minimal pizza on $T$ associated with $f$.
Replacing all arcs $\lambda_\ell$ with some other arcs $\theta_\ell\in D_\ell$, for $\ell=0,\dots,p$,
we see that any choice of arcs $\lambda_\ell\in D_\ell$ results in a minimal pizza on $T$ associated with $f$.

On the other hand, given a minimal pizza $\{T_\ell=T(\lambda_{\ell-1},\lambda_\ell)\}$ on $T$ associated with $f$,
consider any other minimal pizza $\{T'_\ell=T(\theta_{\ell-1},\theta_\ell)\}$ on $T$ associated with $f$.
By the Lipschitz contact invariance of a minimal pizza (see Theorem \ref{pizza-theorem}) there exists an inner bi-Lipschitz homeomorphism $h:T\to T$
such that $h(\lambda_\ell)=\theta_\ell$ and $h(T_\ell)=T'_\ell$ for all $\ell$, and also such that $ord_{h(\gamma)}f=ord_\gamma f$ for any arc $\gamma\subset T$.
Thus $h$ transforms $f$ into a function of the same contact (see Definition 2.2 from \cite{Bi-Fer-Co-Ru}) .
Since the zones $D_\ell$ are Lipschitz invariant, we have $h(D_\ell)=D_\ell$ for all $\ell$, thus $\theta_\ell\in D_\ell$.
This proves that any minimal pizza can be obtained by some choice of arcs $\lambda_\ell\in D_\ell$.
\end{proof}

\begin{corollary}\label{zone-slice} Let $\{D_\ell\}_{\ell=0}^p$ be the pizza zones of a minimal pizza
 $\{T_\ell=T(\lambda_{\ell-1},\lambda_\ell)\}_{\ell=1}^p$ on $T$ associated with $f$, as in Proposition \ref{MP}.
For each $\ell=1,\ldots,p$, the set $Y_\ell=D_{\ell-1}\cup D_\ell\cup V(T_\ell)$ is a pizza slice zone associated with $f$,
independent of the choice of arcs $\lambda_\ell\in D_\ell$.
Moreover, $Y_\ell$ is a maximal pizza slice zone: if a pizza slice zone $Y\subset V(T)$ associated with $f$ contains $Y_\ell$ then $Y=Y_\ell$.
\end{corollary}

\section{Elementary pairs of normally embedded H\"older triangles}\label{Section:LNE triangles}

Let $T=T(\gamma_1,\gamma_2)$ and $T'=T(\gamma'_1,\gamma'_2)$ be two normally embedded $\beta$-H\"older triangles,
oriented from $\gamma_1$ to $\gamma_2$ and from $\gamma'_1$ to $\gamma'_2$ respectively.

\begin{definition}\label{def:regular_pair}\emph{A pair $(\gamma,\gamma')$ of arcs $\gamma\subset T$ and $\gamma'\subset T'$ is \emph{regular} if}
\begin{equation}\label{regularpair}
tord(\gamma,T')=tord(\gamma,\gamma')=tord(\gamma',T).
\end{equation}
\end{definition}

\begin{proposition}\label{map-graph}
 Let $T=T(\gamma_1,\gamma_2)$ and $T'=T(\gamma'_1,\gamma'_2)$ be two normally embedded $\beta$-H\"older triangles.
 Let $f(x)=dist(x,T')$ be the distance from $x\in T$ to $T'$, and let $g(x')=dist(x',T)$ be the distance from $x'\in T'$ to $T$.
 Let $\Gamma\subset T\times\R$ and $\Gamma'\subset T'\times\R$ be the graphs of the functions $f(x)$ and $g(x')$.
Then the following conditions are equivalent:

1. There is a homeomorphism $H:T\cup T'\to T\cup\Gamma$,
bi-Lipschitz with respect to the outer metric, such that $H(\gamma_1)=\gamma_1$ and $H(\gamma_2)=\gamma_2$.

2. There is a homeomorphism $H':T\cup T'\to T'\cup\Gamma'$,
bi-Lipschitz with respect to the outer metric, such that $H'(\gamma'_1)=\gamma'_1$ and $H'(\gamma'_2)=\gamma'_2$.

3. There exists a bi-Lipschitz homeomorphism $h:T\to T'$ such that
$h(\gamma_1)=\gamma'_1$, $h(\gamma_2)=\gamma'_2$ and $tord(\gamma,h(\gamma))=tord(\gamma,T')$ for any arc $\gamma\subset T$.

4. There exists a bi-Lipschitz homeomorphism $h':T'\to T$ such that
$h'(\gamma'_1)=\gamma_1$, $h'(\gamma'_2)=\gamma_2$ and $tord(\gamma',h'(\gamma'))=tord(\gamma',T)$ for any arc $\gamma'\subset T'$.

5. There exists a bi-Lipschitz homeomorphism $h:T\to T'$ such that
$h(\gamma_1)=\gamma'_1$, $h(\gamma_2)=\gamma'_2$, and the pair of arcs $(\gamma,h(\gamma))$ is regular for any arc $\gamma\subset T$.
\end{proposition}

\begin{proof}
If condition 1 is satisfied, we may assume that $H(T)=T$ and $H(T')=\Gamma$.
Since $f$ is a Lipschitz function on $T$ and $H$ is an outer bi-Lipschitz homeomorphism, we have
$tord(\gamma,T')=tord(H(\gamma),\Gamma)=ord_{H(\gamma)} f=tord(H(\gamma),f(H(\gamma))$ for any arc $\gamma\subset T$.
Since $H^{-1}$ is also an outer bi-Lipschitz homeomorphism, the mapping $h:T\to T'$ defined as $h(x)=H^{-1}((H(x),f(H(x)))$ is a bi-Lipschitz homeomorphism satisfying condition 5, which implies conditions 3 and 4.
Conversely, given a homeomorphism $h:T\to T'$ satisfying condition 3, the mapping $H:T\cup T'\to T\cup\Gamma$ which is the identity on $T$
and defined as $H(x')=(h^{-1}{x'},f(h^{-1}(x')))$for $x'\in T'$ satisfies condition 1. Thus conditions 1, 3 and 5 are equivalent.

Similarly, conditions 2, 4 and 5 are equivalent.

If conditions 1 and 3 are satisfied, we may assume that $T'=\Gamma$ and $h(x)=(x,f(x))$ for $x\in T$.
Then $tord(\gamma,T')=tord(\gamma',T)$ for any arcs $\gamma\subset T$ and $\gamma'=\{(x,f(x)):x\in\gamma\}\subset T'$.
Thus $h'=h^{-1}:T'\to T$ satisfies condition 2. This implies that all five conditions are equivalent.
\end{proof}

If conditions of Proposition \ref{map-graph} are satisfied then the pairs of arcs $(\gamma_1,\gamma'_1)$ and $(\gamma_2,\gamma'_2)$ are regular:
 \begin{equation}\label{tord-tord}
tord(\gamma_1,T')=tord(\gamma_1,\gamma'_1)=tord(\gamma'_1,T),\quad tord(\gamma_2,T')=tord(\gamma_2,\gamma'_2)=tord(\gamma'_2,T).
\end{equation}
In general, the opposite does not hold. However, Theorem \ref{theorem:very-elementary} below states that conditions of Proposition \ref{map-graph}
are satisfied if $T$ is elementary with respect to $f$ and (\ref{tord-tord}) holds.
The following Proposition from \cite{GS} is an important step in the proof of Theorem \ref{theorem:very-elementary}.

\begin{proposition}\label{Prop:2.20-GS}\emph{(see \cite[Proposition 2.20]{GS})}
Let $T=T(\gamma_1,\gamma_2)$ and $T'=T(\gamma'_1,\gamma'_2)$ be normally embedded $\beta$-H\"older triangles such that $tord(\gamma_1,\gamma'_1)\ge\alpha,\;tord(\gamma_2,\gamma'_2)\ge\alpha$, and $tord(\gamma,T')\ge\alpha$
for all arcs $\gamma\subset T$, for some $\alpha>\beta$.
Then there is a bi-Lipschitz homeomorphism $h: T\to T'$ such that $h(\gamma_1)=\gamma'_1,\;h(\gamma_2)=\gamma'_2$,
and $tord(h(\gamma),\gamma)\ge\alpha$ for any arc $\gamma\subset T$.
\end{proposition}

\begin{remark}\label{rem:beta}\normalfont
If a $\beta$-H\"older triangle $T=T(\gamma_1,\gamma_2)$ and a $\beta'$-H\"older triangle
$T'=T(\gamma'_1,\gamma'_2)$ are normally embedded and satisfy (\ref{tord-tord}) then $\beta'=\beta$, unless $tord(T,T')\le\min(\beta,\beta')$.
\end{remark}

\begin{lemma}\label{lem:min-contact} Let $T=T(\gamma_1,\gamma_2)$ and $T'=T(\gamma'_1,\gamma'_2)$ be normally embedded H\"older triangles.
Let $\lambda_1\ne\lambda_2$ be two arcs in $T$, and let $\theta_1\subset T',\;\theta_2\subset T'$ and $\theta\subset T(\theta_1,\theta_2)\subset T'$ be three arcs such that $tord(\theta,T)=q<\min(tord(\lambda_1,\theta_1),tord(\lambda_2,\theta_2))$.
Then there is an arc $\lambda\subset T(\lambda_1,\lambda_2)\subset T$ such that $tord(\lambda,T')\le q$.
\end{lemma}

\begin{proof}
We may assume that $\theta_1\subset T'_1=T(\gamma'_1,\theta)$ and $\theta_2\subset T'_2=T(\theta,\gamma'_2)$.
For $x\in T$, let $f_1(x)=dist(x,T'_1)$ and $f_2(x)=dist(x,T'_2)$. Then $f(x)=dist(x,T')=\min(f_1(x),f_2(x))$.
Since $tord(\theta,\lambda_1)\le tord(\theta,T)=q$ and $tord(\lambda_1,\theta_1) > q$, we have by the non-Archimedean property
$tord(\theta_1,\theta)=\min(tord(\lambda_1,\theta_1),tord(\lambda_1,\theta))\le q$.
Since $T'$ is normally embedded, we have $tord(\theta_1,T'_2)=tord(\theta_1,\theta)\le q$.
Since $tord(\lambda_1,\theta_1)>q$, this implies $tord(\lambda_1,T'_2)=\min(tord(\lambda_1,\theta_1),tord(\theta_1,T'_2))\le q$ by the non-Archimedean property.
Thus $f|_{\lambda_1}=f_1|_{\lambda_1}$.
Similarly,  $f|_{\lambda_2}=f_2|_{\lambda_2}$, thus there is an arc $\lambda\subset T(\lambda_1,\lambda_2)$ such that
$f|_\lambda=f_1|_\lambda=f_2|_\lambda$ (see Fig.~\ref{fig:min-contact}).
Then $tord(\lambda,T')=ord_\lambda f\le q$, otherwise we would have $tord(\lambda,T'_1)=tord(\lambda,T'_2)>q$.
Since $tord(\lambda,\theta)\le tord(\theta,T)=q$, this would contradict to $T'=T'_1\cup T'_2$ being normally embedded.
\end{proof}

\begin{corollary}\label{long-zone}
Let $T$ and $T' $ be normally embedded H\"older triangles. Let $\tilde T=T(\lambda_1,\lambda_2)\subset T$ be a $\beta$-H\"older triangle
such that $tord(\gamma,T')=q>\beta$ for any arc $\gamma\subset\tilde T$. If $\tilde T'=T(\theta_1,\theta_2)\subset T'$ is a $\beta$-H\"older
triangle such that $tord(\theta_1,\lambda_1)=tord(\theta_2,\lambda_2)=q$ then, for any arc $\theta\subset T'$ such that $tord(\theta,\theta_1)<q$ and
$tord(\theta,\theta_2)<q$, we have $tord(\theta,T)=q$.
\end{corollary}

\begin{proof} Lemma \ref{lem:min-contact} implies that $tord(\theta,T)\ge tord(\theta,\tilde T)\ge q$ for any arc $\theta\subset\tilde T'$.
If $\theta\subset\tilde T'$ is an arc such that $tord(\theta,\theta_1)<q$ and $tord(\theta,\theta_2)<q$,
Proposition \ref{Prop:2.20-GS} implies that $tord(\theta,T\setminus\tilde T)<q$.
If $tord(\theta,T)>q$ and $\gamma\subset T$ is an arc such that $tord(\gamma,\theta)>q$,
then $\gamma\subset\tilde T$ and $tord(\gamma,T')>q$, a contradiction. Thus $tord(\theta,T)=q$.
\end{proof}

\begin{definition}\label{def:triangles-oriented}\normalfont
Let $T$ and $T'$ be normally embedded oriented H\"older triangles.
A pair of $\beta$-H\"older triangles $\tilde T=T(\lambda_1,\lambda_2)\subset T$ and $\tilde T'=T(\theta_1,\theta_2)\subset T'$ in Corollary \ref{long-zone}
is called \emph{positively oriented} if their orientations induced from $T$ and $T'$ are either both the same as their orientations
from $\lambda_1$ to $\lambda_2$ and from $\theta_1$ to $\theta_2$ or both opposite.
Otherwise, $\tilde T$ and $\tilde T'$ is called a \emph{negatively oriented} pair of H\"older triangles.
\end{definition}

\begin{remark}\label{rem:oriented}\normalfont
Any pair of $\alpha$-H\"older triangles $T(\lambda'_1,\lambda'_2)\subset\tilde T$ and $T(\theta'_1,\theta'_2)\subset\tilde T'$, where $\alpha<q$, satisfying conditions $tord(\theta'_1,\lambda'_1)=tord(\theta'_2,\lambda'_2)=q$ is positively (resp., negatively) oriented if, and only if, the pair $(\tilde T,\tilde T')$ is positively (resp., negatively) oriented.
\end{remark}

\begin{proposition}\label{long-q-zone}
Let $T$ and $T' $ be normally embedded H\"older triangles with the distance functions $f(x)=dist(x,T')$ and $g(x')=dist(x',T)$.
Let $Z\subset V(T)$ be a maximal $q$-order zone for $f$ such that $\mu(Z)<q$.
Then there exists a unique maximal $q$-order zone $Z'\subset V(T')$ for $g$ such that $\mu(Z')=\mu(Z)$ and,
for any arc $\gamma\in Z$ such that $\nu_Z(\gamma,f)<q$ and any arc $\gamma'\subset T'$
such that $tord(\gamma,\gamma')=q$, we have $\gamma'\subset Z'$ and $\nu_{Z'}(\gamma',g)=\nu_Z(\gamma,f)$.
Conversely, if $\gamma'\in Z'$ is any arc such that $\nu_{Z'}(\gamma',g)<q$ then, for any arc $\gamma\subset T$
such that $tord(\gamma,\gamma')=q$, we have $\gamma\subset Z$ and $\nu_Z(\gamma,f)=\nu_{Z'}(\gamma',g)$.
\end{proposition}

\begin{proof}
Let $\tilde Z\subset Z$ be the set of all arcs $\gamma\subset Z$ such that $\nu_Z(\gamma,f)<q$.
Let $\lambda_1$ and $\lambda_2$ be any two arcs in $\tilde Z$ such that $\beta=tord(\lambda_1,\lambda_2)<q$.
Consider the $\beta$-H\"older triangle $\tilde T=T(\lambda_1,\lambda_2)\subset T$.
Since $Z$ is a zone, we have $V(\tilde T)\subset Z$, thus $tord(\gamma,T')=q$ for any arc $\gamma\subset\tilde T$.
Also, $\nu_Z(\gamma,f)\le\max(\nu_Z(\lambda_1,f),\nu_Z(\lambda_2,f))<q$ for any arc $\gamma\subset\tilde T$, thus $V(\tilde T)\subset\tilde Z$.
This implies that $\tilde Z$ is a $q$-order zone for $f$.
Let $\tilde Z'$ be the set of all arcs $\gamma'\subset T'$ such that $tord(\gamma,\gamma')=q$ for some arc $\gamma\subset\tilde Z$.

Let $\lambda'_1$ and $\lambda'_2$ be any two arcs in $T'$ such that $tord(\lambda_1,\lambda'_1)=tord(\lambda_2,\lambda'_2)=q$.
Since $\beta<q$, $\tilde T'=T(\lambda'_1,\lambda'_2)\subset T'$ is a $\beta$-H\"older triangle.
Since $\lambda_1\in\tilde Z$ and $\lambda_2\in\tilde Z$, we have $V(\tilde T')\subset\tilde Z'$.
It follows from Proposition \ref{Prop:2.20-GS} that $tord(\gamma,T'\setminus\tilde T')<q$ for any arc $\gamma\subset\tilde T$,
thus $tord(\gamma,\tilde T')=q$ for any arc $\gamma\subset\tilde T$.
Corollary \ref{long-zone} implies that $tord(\theta,\tilde T)=q$ for any arc $\theta\subset\tilde T'$.
This implies that $\tilde Z'$ is a $q$-order zone for $g$.
It follows from the non-Archimedean property that $\nu_{\tilde Z'}(\gamma',g)<q$ for any arc $\gamma'\in\tilde Z'$.

Let $Z'$ be the maximal $q$-zone for $g$ containing $\tilde Z'$. By the construction this zone is unique.
Let us show that $\nu_{Z'}(\gamma',g)\ge q$ for any arc $\gamma'\in Z'\setminus\tilde Z'$.
If $\gamma'\in Z'\setminus\tilde Z'$ and $\nu_{Z'}(\gamma',g)<q$, applying the same arguments as above to $Z'$ and $g$
instead of $Z$ and $f$, we can show that any arc $\gamma\subset T$ such that $tord(\gamma,\gamma')=q$ belongs to $Z$
and $\nu_Z(\gamma,f)<q$. Thus $\gamma\in\tilde Z$, which implies $\gamma'\in\tilde Z'$, a contradiction.
The equality $\nu_Z(\gamma,f)=\nu_{Z'}(\gamma',g)$ follows from the non-Archimedean property.
\end{proof}

\begin{corollary}\label{correspondance}
Let $T$ and $T' $ be normally embedded H\"older triangles with the distance functions $f(x)=dist(x,T')$ and $g(x')=dist(x',T)$.
For any $q\in\F$, the finite set $L_q$ of maximal $q$-order zones $Z\subset V(T)$ for $f$ such that $\mu(Z)<q$ is nonempty if, and only if, the set
$L'_q$ of maximal $q$-order zones $Z'\subset V(T')$ for $g$ such that $\mu(Z')<q$ is nonempty, and
there is a canonical one-to-one correspondence $Z'=\tau_q(Z)$ between the sets $L_q$ and $L'_q$ such that $tord(Z,\tau_q(Z))=q$.
\end{corollary}

\begin{proof}
The finiteness of the set $L_q$ follows from the fact that, for a given $q\in\F$, each pizza slice of a pizza on $T$ associated with $f$ contains at most one zone from $L_q$.
\end{proof}

\begin{lemma}\label{oriented-zone}
Let $T$ and $T'$ be normally embedded oriented H\"older triangles.
Let $Z\subset V(T)$ and $Z'\subset V(T')$ be maximal $q$-order zones for $f$ and $g$ respectively, of orders $\mu(Z)=\mu(Z')<q$,
related as in Proposition \ref{long-q-zone} and Corollary \ref{correspondance}.
Then the pairs $(\tilde T,\tilde T')$ of H\"older triangles $\tilde T\subset T$ and $\tilde T'\subset T'$ related as in Corollary \ref{long-zone}, such that $V(\tilde T)\subset Z$ and $V(\tilde T')\subset Z'$, are either all positively oriented or all negatively oriented.
\end{lemma}

\begin{proof} This follows from Remark \ref{rem:oriented}, since for any two pairs of H\"older triangles in Lemma \ref{oriented-zone}
there is a larger pair of H\"older triangles containing both of them and satisfying conditions of Corollary \ref{long-zone}.
\end{proof}

\begin{definition}\label{def:zones-oriented}\normalfont
The pair of zones $Z\subset V(T)$ and $Z'\subset V(T')$ in Lemma \ref{oriented-zone} is called \emph{positively oriented} (resp., \emph{negatively oriented}) if the pairs $(\tilde T,\tilde T')$ of H\"older triangles $\tilde T\subset T$ and $\tilde T'\subset T'$ in
Lemma \ref{oriented-zone} are positively oriented (resp., negatively oriented).
\end{definition}

\begin{figure}
\centering
\includegraphics[width=5in]{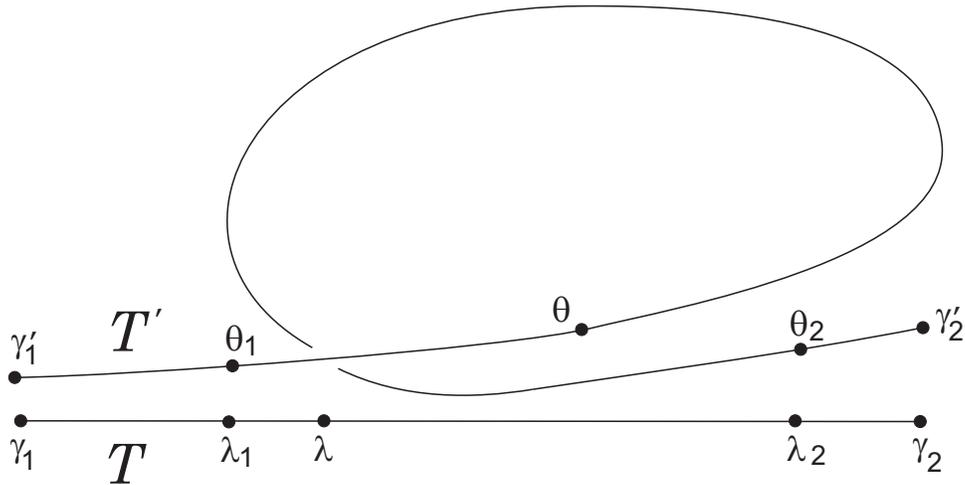}
\caption{Illustration to the proof of Lemma \ref{lem:min-contact}.}\label{fig:min-contact}
\end{figure}
\begin{figure}
\centering
\includegraphics[width=5in]{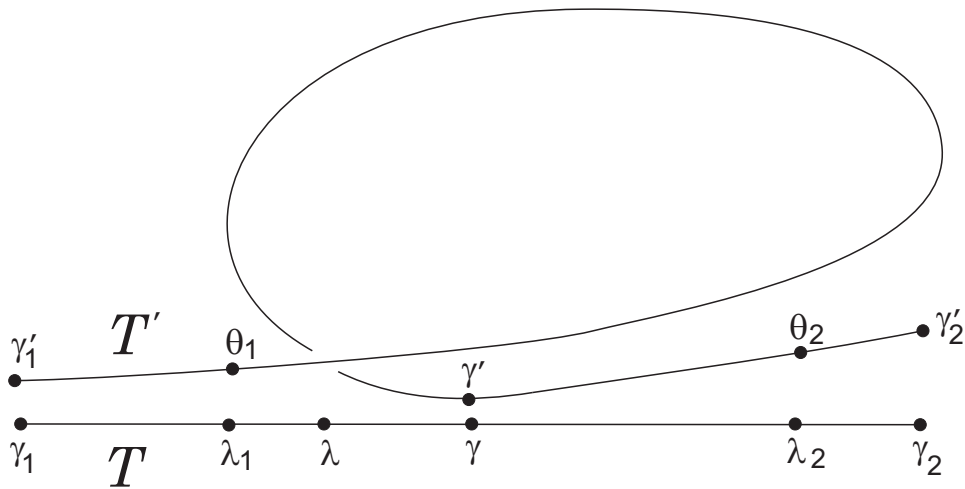}
\caption{Illustration to the proof of Lemma \ref{lem:elementary}.}\label{fig:elementary}
\end{figure}

\begin{lemma}\label{lem:elementary}
Let $T=T(\gamma_1,\gamma_2)$ and $T'=T(\gamma'_1,\gamma'_2)$ be two normally embedded H\"older triangles, such that
$T$ is elementary with respect to $f(x)=dist(x,T')$ and $tord(\gamma_1,\gamma'_1)=tord(T,T')$.
Then $T'$ is elementary with respect to $g(x')=dist(x',T)$.
\end{lemma}

\begin{proof} We have to show that, for any H\"older triangle $T''=T(\theta_1,\theta_2)\subset T'$
such that $tord(\theta_1,T)=tord(\theta_2,T)=q$, we have $tord(\gamma',T)=q$ for each arc $\gamma'\subset T''$.
Let us show first that $q'=tord(\gamma',T)\ge q$ for each arc $\gamma'\subset T''$.
If $q'<q$, let $\lambda_1$ and $\lambda_2$ be arcs in $T$ such that $tord(\lambda_1,\theta_1)=tord(\lambda_2,\theta_2)=q$.
Lemma \ref{lem:min-contact} implies that there is an arc $\lambda\subset T(\lambda_1,\lambda_2)$ such that
$tord(\lambda,T')\le q'<q$, a contradiction with $T$ being elementary with respect to $f$.

Suppose now that $q'>q$.
We may assume that $\theta_1\subset T(\gamma'_1,\gamma')\subset T'$.
Since $tord(\gamma_1,\gamma'_1)=tord(T,T')$, we have $tord(\gamma_1,\gamma'_1)\ge q'$.
Let $\gamma\subset T$ be an arc such that $tord(\gamma,\gamma')=q'$ (see Fig.~\ref{fig:elementary}).
Then Lemma \ref{lem:min-contact} applied to $T(\gamma_1,\gamma)\subset T$ and $T(\gamma'_1,\gamma')\subset T'$ implies that there is an arc
$\lambda\subset T(\gamma_1,\gamma)$ such that $tord(\lambda,T')\le q$,  a contradiction with $T$ being elementary with respect to $f$.
\end{proof}

\begin{corollary}\label{elementary-elementary}
Let $T=T(\gamma_1,\gamma_2)$ and $T'=T(\gamma'_1,\gamma'_2)$ be normally embedded H\"older triangles, such that
$T$ is elementary with respect to $f(x)=dist(x,T')$ and $tord(\gamma_1,\gamma'_1)=tord(T,T')$.
Then, for any two H\"older triangles $\tilde T=T(\gamma_1,\lambda)\subset T$ and $\tilde T'=T(\gamma'_1,\lambda')\subset T'$,
$\tilde T$ is elementary with respect to $\tilde f(x)=dist(x,\tilde T')$ and
$\tilde T'$ is elementary with respect to $\tilde g(x')=dist(x',\tilde T)$.
\end{corollary}

\begin{proof}
Since $\tilde T$ is elementary with respect to $f|_{\tilde T}$, Lemma \ref{lem:elementary} applied to $\tilde T$ instead of $T$ implies that
$T'$ is elementary with respect to $h(x')=dist(x',\tilde T)$.
Thus $\tilde T'$ is elementary with respect to $\tilde g=h|_{\tilde T'}$.
Lemma \ref{lem:elementary} applied to $\tilde T'$ instead of $T$ and $\tilde T$ instead of $T'$
implies that $\tilde T$ is elementary with respect to $\tilde f$.
\end{proof}

\begin{lemma}\label{lem:pizzaslice}
Let $T=T(\gamma_1,\gamma_2)$ and $T'=T(\gamma'_1,\gamma'_2)$ be normally embedded $\beta$-H\"older triangles satisfying (\ref{tord-tord}).
Suppose that $T$ is a pizza slice associated with $f(x)=dist(x,T')$.
Then conditions of Proposition \ref{map-graph} are satisfied for $T$ and $T'$.
Moreover,  $\mu_{T,f}\equiv\mu_{T',g}$, where $\mu_{T,f}(q)$ and $\mu_{T',g}(q)$ are the width functions defined on $Q_f(T)=Q_g(T')$.
\end{lemma}

\begin{proof}
Since the five conditions of Proposition \ref{map-graph} are equivalent, it is enough to prove condition 3:
there is a bi-Lipschitz homeomorphism $h:T\to T'$ such that
$h(\gamma_1)=\gamma'_1$, $h(\gamma_2)=\gamma'_2$ and $tord(\gamma,h(\gamma))=tord(\gamma,T')$ for each arc $\gamma\subset T$.

Let $Q=Q_f(T)$, and let the width function $\mu(q)=\mu_{T,f}(q):Q\to\F\cup\{\infty\}$ be affine, $\mu(q)=aq+b$.
We consider the following cases: (1) $Q=\{\alpha\}$ where $\alpha\le\beta$, (2) $Q=\{\alpha\}$ where $\alpha>\beta$, (3) $\mu(q)\equiv q$, (4)
$\mu(q)<q$ for all $q\in Q$, (5) $\mu(q)=q$ only for the maximal value of $\mu(q)$, (6) $\mu(q)=q$ only for the minimal value $\mu(q)=\beta$.

Case 1. Any bi-Lipschitz homeomorphism  $h:T\to T'$ such that $h(\gamma_1)=\gamma'_1$ and $h(\gamma_2)=\gamma'_2$ satisfies
$tord(\gamma,h(\gamma))=tord(\gamma,T')=\alpha$ for all arcs $\gamma\in V(T)$.

Case 2. It follows from \cite[Proposition 2.20]{GS} (see Proposition \ref{Prop:2.20-GS}) that there is a bi-Lipschitz
homeomorphism $h:T\to T'$ such that $tord(\gamma,h(\gamma))\ge\alpha$ for any arc $\gamma\subset T$.
Since $tord(\gamma,h(\gamma))\le tord(\gamma,T')=\alpha$ for any arc $\gamma\subset T$, we have $tord(\gamma,h(\gamma))=\alpha$.

Case 3. We may assume that $Q$ is not a point and $q_1=tord(\gamma_1,T')$ is the maximal value of $q\in Q$.
Then $q=ord_\gamma f=\mu_T(\gamma,f)=tord(\gamma,\gamma_1)$ for all arcs $\gamma\subset T$ such that $tord(\gamma,\gamma_1)\le q_1$, otherwise
$ord_\gamma f=q_1$.
Any bi-Lipschitz homeomorphism  $h:T\to T'$ such that $h(\gamma_1)=\gamma'_1$ satisfies
$tord(h(\gamma),\gamma'_1)=tord(\gamma,\gamma_1)$ for all arcs $\gamma\subset T$.
Thus $q=ord_\gamma f=\mu_T(\gamma,f)=tord(\gamma,\gamma_1)=tord(h(\gamma),\gamma'_1)$ for all $\gamma\subset T$ such that $tord(\gamma,\gamma_1)\le q_1$.
Since $tord(\gamma_1,\gamma'_1)=q_1\ge q$, this implies that $tord(\gamma,h(\gamma))\ge q$.
If $tord(\gamma,h(\gamma))>q$ then $tord(\gamma,T')>q$, a contradiction. Thus  $tord(\gamma,h(\gamma))=q$ for all $\gamma\subset T$ such that $tord(\gamma,\gamma_1)\le q_1$. Otherwise, if $tord(\gamma,\gamma_1)>q_1$, then $tord(h(\gamma),\gamma'_1)>q_1$, thus
$tord(\gamma,h(\gamma))=tord(\gamma_1,\gamma'_1)=q_1=tord(\gamma,T')$.

Case 4. Using the same arguments as in the proof of \cite[Proposition 2.20]{GS},
we assume that $T'=T_\beta\subset \R^2$ is a standard $\beta$-H\"older triangle (\ref{Formula:Standard Holder triangle}),
$T\cup T'\subset\R^n$, and $\pi:T\to\R^2$ is an orthogonal projection.
We may also assume that $Q$ is not a point, and that $\mu(q_1)$, where $q_1=tord(\gamma_1,T')$, is the maximal value of $\mu(q)$ for $q\in Q$.
Then $\mu_T(\gamma,f)=tord(\gamma,\gamma_1)$ for all arcs $\gamma\subset T$ such that $tord(\gamma,\gamma_1)\le \mu(q_1)$, otherwise
$ord_\gamma f=q_1$.

The set $S\subset T$ where $\pi$ is not smooth and orientation-preserving is a finite union of isolated arcs and $\beta_j$-H\"older
triangles $T_j=T(\lambda_j,\lambda'_j)\subset T$.
We want to show that $f$ has the same order $q_j$ on each arc $\gamma\subset T_j$.
It is enough to show that $ord_{\lambda_j} f=ord_{\lambda'_j} f$.
We may assume that $\lambda'_j\subset T(\gamma_1,\lambda_j)\subset T$, thus $\mu_j=\mu_T(\lambda_j,f)\le\mu(\lambda'_j,f)$.
If $ord_{\lambda_j} f=q_j\ne ord_{\lambda'_j} f$ then $\beta_j\le\mu_j=tord(\lambda_j,\gamma_1)$.
Since $T_j$ is orientation-reversing and $T$ is normally embedded, we have $\beta_j\ge q_j$, a contradiction with the condition $\mu_j<q_j$.
Thus $ord_{\lambda_j} f=ord_{\lambda'_j} f=q_j\le\beta_j<\mu_j$,
and there is a $\mu_j$-H\"older triangle $\tilde T_j\subset T$ containing $T_j$ such that
$ord_\gamma f=q_j$ for each arc $\gamma\subset \tilde T_j$.

It follows from \cite[Proposition 2.20]{GS} that there is a bi-Lipschitz orientation-preserving homeomorphism $h_j:\tilde T_j\to \pi(\tilde T_j)\cap T'$
such that $tord(\gamma,h_j(\gamma))=q_j$ for each arc $\gamma\subset\tilde T_j$.
One can choose triangles $\tilde T_j$ so that they are all disjoint.
Replacing projection $\pi$ with the homeomorphisms $h_j$ on each triangle $\tilde T_j$, a bi-Lipschitz homeomorphism $h:T\to T'$ can be obtained,
such that $tord(\gamma,h(\gamma))=tord(\gamma,T')$ for each arc $\gamma\subset T$.

Case 5. Assuming that $\mu(q_1)=q_1=tord(\gamma_1,\gamma'_1)$ is the maximal value of $\mu(q)$,
for any arc $\gamma\subset T$ such that $tord(\gamma,T')=q_1$ we have $tord(\gamma,\gamma_1)\ge\mu(\gamma)=q_1$, thus $tord(\gamma,h(\gamma))=tord(\gamma_1,\gamma'_1)=q_1$.
For any arc $\gamma\subset T$ such that $q=tord(\gamma,T')>\mu(q)=tord(\gamma,\gamma_1)$, the same arguments as in Case 4 apply.

Case 6. The same arguments as in Case 4 imply that, for any triangle $T_j\subset T$ containing an arc $\gamma$ such that $ord_\gamma f>\beta$
and $\pi|_T$ is orientation-reversing, the order $q_j$ of $f$ is the same on all arcs of $T_j$, and there is a $\mu_j$-triangle $\tilde T_j$
containing $T_j$, where $\mu_j>\beta$, such that $\pi|_{\tilde T_j}$ can be replaced with a bi-Lipschitz orientation-preserving homeomorphism
$h_j:\tilde T_j\to \pi(\tilde T_j)\cap T'$, such that $tord(\gamma,h_j(\gamma))=q_j$ for each arc $\gamma\subset\tilde T_j$.
This allows one to find $\beta$-H\"older triangles $\tilde T=T(\gamma_1,\tilde\gamma)\subset T$ and $\tilde T'=T(\gamma'_1,\tilde\gamma')\subset T'$ such that $\bar T=T(\tilde\gamma,\gamma_2)\subset T$ and $\bar T'=T(\tilde\gamma',\gamma'_2)\subset T'$ are also $\beta$-H\"older triangles,
$ord_\gamma f=\beta$ for each arc $\gamma\subset\bar T$, and to obtain
a bi-Lipschitz homeomorphism $\tilde h:\tilde T\to \tilde T'$ such that $tord(\gamma,h(\gamma))=tord(\gamma,T')$ for each arc $\gamma\subset \tilde T$. After that, $\tilde h$ combined with any bi-Lipschitz homeomorphism $\bar h:\bar T\to\bar T'$, such that
$\bar h(\tilde\gamma)=\tilde\gamma'$ and $\bar h(\gamma_2)=\gamma'_2$, defines a bi-Lipschitz homeomorphism  $h:T\to T'$ such that $tord(\gamma,h(\gamma))=tord(\gamma,T')$ for each arc $\gamma\subset T$.

The existence of a mapping $h:T\to T'$ satisfying condition 5 of Proposition \ref{map-graph} implies that $Q_f(T)=Q_g(T')$ and
 $\mu_{T,f}\equiv\mu_{T',g}$.
\end{proof}

\begin{definition}\label{def:pizzaslice-oriented}\normalfont
Let $T=T(\gamma_1,\gamma_2)$ and $T'=T(\gamma'_1,\gamma'_2)$ be two normally embedded oriented $\beta$-H\"older triangles satisfying (\ref{tord-tord}), such that $T$ is a pizza slice associated with $f(x)=dist(x,T')$ and $tord(T,T')=tord(\gamma_1,\gamma'_1)>\beta$.
The pair $(T,T')$ is called \emph{positively oriented} if either $T$ is oriented from $\gamma_1$ to $\gamma_2$ and $T'$ from $\gamma'_1$ to $\gamma'_2$,
or $T$ is oriented from $\gamma_2$ to $\gamma_1$ and $T'$ from $\gamma'_2$ to $\gamma'_1$.
Otherwise, the pair $(T,T')$ is called \emph{negatively oriented}.
\end{definition}

\begin{lemma}\label{lem:slices-oriented}
Let $T=T(\gamma_1,\gamma_2)$ and $T'=T(\gamma'_1,\gamma'_2)$ be two normally embedded $\beta$-H\"older triangles in Definition \ref{def:pizzaslice-oriented}
such that $\mu_{T,f}(q)\not\equiv q$. For $q\in Q_T(f)$ such that $\mu(q)<q$, let $Z_q\subset V(T)$ and $Z'_q\subset V(T')$ be the maximal $q$-order zones
for $f(x)=dist(x,T')$ and $g(x')=dist(x',T)$ respectively. Then the pair of zones $(Z_q,Z'_q)$ is positively oriented if the pair of H\"older triangles $(T,T')$ is positively oriented, and negatively oriented otherwise.
\end{lemma}

\begin{lemma}\label{lem:tord}
Let a $\beta$-H\"older triangle $T=T(\gamma_1,\gamma_2)$ and a $\beta'$-H\"older triangle $T'=T(\gamma'_1,\gamma'_2)$ be normally embedded, where
$\beta\ge\beta',\; q_1=tord(\gamma_1,\gamma'_1)=tord(T,T')>\beta$ and $q_2=tord(\gamma_2,T')\ge\beta$.
If $T$ is a pizza slice associated with $f(x)=dist(x,T')$ then there is an arc $\theta\subset T'$ such that
$tord(\gamma'_1,\theta)=\beta$, and
\begin{equation}\label{tord2}
tord(\gamma_2,\theta)=tord(\theta,T)=q_2,
\end{equation}
thus conditions (\ref{tord-tord}) are satisfied for triangles $T$ and $T(\gamma'_1,\theta)\subset T'$.
Moreover, if $q_2>\beta$ then $tord(\gamma'_1,\theta)=\beta$ for any arc $\theta\subset T'$ satisfying condition (\ref{tord2}),
and if $q_2=\beta$ then any arc $\theta\subset T'$ such that $tord(\gamma'_1,\theta)=\beta$ satisfies condition (\ref{tord2}).
\end{lemma}

\begin{proof} Let $\theta\subset T'$ be an arc such that $tord(\gamma_2,\theta)=q_2$.
Note first that $\alpha=tord(\theta,T)\ge q_2$.
Suppose that $\alpha>q_2$, and let $\lambda\subset T$ be an arc such that $tord(\lambda,\theta)=\alpha$,thus $q'=tord(\lambda,T')\ge\alpha>q_2$.
Let $\mu_T(q)$ be the affine width function of $T$. Note that $\mu$ cannot be constant, since $q_1\ge\alpha >q_2$,
If the maximum of $\mu_T(q)$ is at $q=q_2$ then $tord(\lambda,\gamma_2)=\mu(q')\le\mu(\alpha)<\mu(q_2)\le q_2$.
Since $q_2=tord(\gamma_2,\theta)$ and $\alpha=tord(\lambda,\theta)>q_2$, this contradicts the non-Archimedean property.
Thus the maximum of $\mu_T(q)$ is at $q=q_1$ and its minimum is $\mu(q_2)=\beta$.

If $q_2>\beta$ then $tord(\theta,\gamma'_1)=tord(\gamma_1,\gamma_2)=\beta$ for any arc $\theta\subset T'$ satisfying condition (\ref{tord2}).
However, since $tord(\gamma_1,\lambda)>\beta$, $tord(\gamma_1,\gamma'_1)=\mu(q_1)>\beta$ and $\alpha=tord(\lambda,\theta)>q_2\ge\beta$, condition $tord(\gamma'_1,\theta)=\beta$ cannot be satisfied, a contradiction. Thus $\alpha=q_2$ in this case.

Otherwise, if $q_2=\beta$, then any arc $\theta\subset T'$ such that $\alpha=tord(\theta,T)>\beta$
satisfies $tord(\theta,\gamma'_1)\ge\min(\alpha,q_1)>\beta$, thus any arc $\theta\subset T'$ such that $tord(\gamma'_1,\theta)=\beta$ satisfies condition (\ref{tord2}).
\end{proof}

\begin{proposition}\label{symmetric-pizza}
Let $T=T(\gamma_1,\gamma_2)$ and $T'=T(\gamma'_1,\gamma'_2)$ be normally embedded $\beta$-H\"older triangles with the distance functions $f(x)=dist(x,T')$ and $g(x')=dist(x',T)$, such that $T$ is elementary with respect to $f$ and $tord(\gamma_1,\gamma'_1)=tord(T,T')$.
Let $\lambda\subset T$ and $\lambda'\subset T'$ be a regular pair of arcs such that $\tilde T=T(\gamma_1,\lambda)$ and $\tilde T'=T(\gamma'_1,\lambda')$ are $\tilde\beta$-H\"older triangles and $\check T=T(\lambda,\gamma_2)$ and $\check T'=T(\lambda',\gamma'_2)$ are $\check\beta$-H\"older triangles.
If both pairs $(\tilde T,\tilde T')$ and $(\check T,\check T')$ satisfy conditions of Proposition \ref{map-graph},
then the pair $(T,T')$ satisfies conditions of Proposition \ref{map-graph}.
\end{proposition}

\begin{proof}
Since conditions 1 - 5 of Proposition \ref{map-graph} are equivalent, it is enough to prove condition 3, i.e., to find a bi-Lipschitz homeomorphism
$h:T\to T'$ such that $h(\gamma_1)=\gamma'_1$, $h(\gamma_2)=\gamma'_2$ and $tord(\gamma,h(\gamma))=tord(\gamma,T')$ for each arc $\gamma\subset T$.
Conditions of Proposition \ref{symmetric-pizza} imply that there is a bi-Lipschitz homeomorphism
$\tilde h:\tilde T\to \tilde T'$ such that $\tilde h(\gamma_1)=\gamma'_1$, $\tilde h(\lambda)=\lambda'$ and $tord(\gamma,\tilde h(\gamma))=tord(\gamma,\tilde T')=tord(\tilde h(\gamma,\tilde T)$ for each arc $\gamma\subset\tilde T$, and also a bi-Lipschitz homeomorphism
$\check h:\check T\to \check T'$ such that $\check h(\lambda)=\lambda'$, $\check h(\gamma_2)=\gamma'_2$ and $tord(\gamma,\check h(\gamma))=tord(\gamma,\check T')=tord(\check h(\gamma),\check T)$ for each arc $\gamma\subset \check T$.
We may assume that $\tilde h(x)=\check h(x)$ for $x\in\lambda$.

We claim that a bi-Lipschitz homeomorphism $h:T\to T'$ with the necessary properties can be defined as $h(x)=\tilde h(x)$ for $x\in\tilde T$
and $h(x)=\check h(x)$ for $x\in\check T$. It is enough to show that $tord(\gamma,T')=tord(\gamma,\tilde T')$ for any arc $\gamma\subset\tilde T$
and $tord(\gamma,T')=tord(\gamma,\check T')$ for any arc $\gamma\subset\check T$.

Since $T$ is elementary with respect to $f$, Lemma \ref{lem:elementary} implies that $T'$ is elementary with respect to $g$.
Corollary \ref{elementary-elementary} implies that $\tilde T$ is elementary with respect to $\tilde f$,
$\tilde T'$ is elementary with respect to $\tilde g$,  $\check T$ is elementary with respect to $\check f$
and $\check T'$ is elementary with respect to $\check g$.

Let $\gamma\subset \tilde T$. Then $\alpha=tord(\gamma,T')=max(tord(\gamma,\tilde T'),tord(\gamma,\check T')\ge tord(\gamma,\tilde T')$.
If $\alpha >tord(\gamma,\tilde T')$ then there is an arc $\gamma'\subset\check T'$ such that $tord(\gamma,\gamma')=tord(\gamma,\check T')=\alpha$,
thus $tord(\gamma',T)\ge\alpha$.
Then
$$\alpha>tord(\gamma,\tilde h(\gamma))=tord(\tilde h(\gamma),\tilde T)\ge tord(\lambda',\tilde T)=tord(\lambda,\lambda')=tord(\lambda',\check T)$$
Thus $tord(\gamma',T)>tord(\lambda',T)$, a contradiction with $T'$ being elementary with respect to $g$.

Let $\gamma\subset \check T$. Then $\alpha=tord(\gamma,T')\le tord(\lambda,T')=tord(\lambda,\lambda')$.
If $\alpha>tord(\gamma,\check T')$ then there is an arc $\gamma'\subset \tilde T'$ such that
$tord(\gamma,\gamma')=\alpha>tord(\gamma,\check T'\ge tord(\gamma,\lambda')$ (see Fig.~\ref{fig:symmetric-pizza}).
This implies that $tord(\gamma',\lambda')=tord(\gamma,\lambda')$.
Since $tord(\tilde h^{-1}(\gamma'),\gamma)\le tord(\tilde h^{-1}(\gamma'),\lambda)=tord(\gamma',\lambda')$,
we have $tord(\tilde h^{-1}(\gamma'),\gamma')=tord(\tilde h^{-1}(\gamma'),\gamma)<\alpha$.
Since $T$ is elementary with respect to $f$, we have $tord(\tilde h^{-1}(\gamma'),\gamma')\ge tord(\lambda,\lambda')\ge\alpha$,
a contradiction with condition $tord(\tilde h^{-1}(\gamma'),\gamma)<\alpha$.
Thus $tord(\gamma,T')=tord(\gamma,\check T)$.
 \end{proof}

 \begin{figure}
\centering
\includegraphics[width=5in]{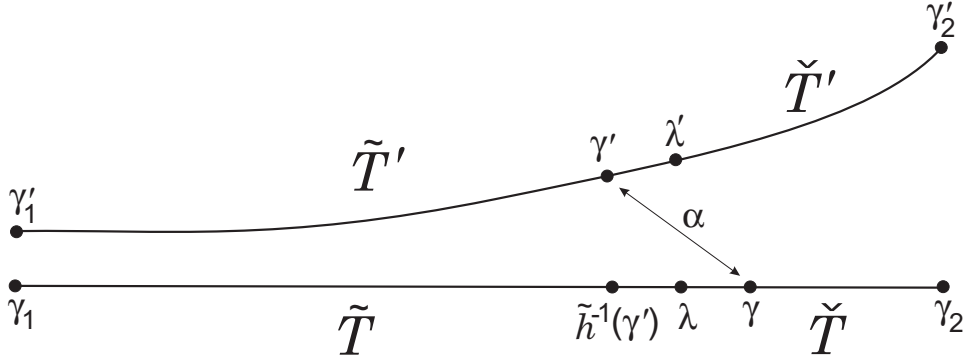}
\caption{Illustration to the proof of Proposition \ref{symmetric-pizza}.}\label{fig:symmetric-pizza}
\end{figure}

\begin{theorem}\label{theorem:very-elementary}
Let $T=T(\gamma_1,\gamma_2)$ and $T'=T(\gamma'_1,\gamma'_2)$ be normally embedded $\beta$-H\"older triangles satisfying (\ref{tord-tord}).
Let  $f(x)=dist(x,T')$ for $x\in T$ and $g(x')=dist(x',T)$ for $x'\in T'$.
If $T$ is elementary with respect to $f$ then $T$ and $T'$ satisfy conditions of Proposition \ref{map-graph}.
Moreover,  $\mu_{T,f}\equiv\mu_{T',g}$, where $\mu_{T,f}(q)$ and $\mu_{T',g}(q)$ are the width functions defined on $Q_f(T)=Q_g(T')$.
\end{theorem}

\begin{proof}
Since the triangles are elementary, we may assume that $tord(\gamma_1,\gamma'_1)=tord(T,T')$.
Let $\{T_i\}_{i=1}^p$ be a minimal pizza decomposition of $T$ associated with $f$, where each pizza slice
$T_i=T(\lambda_{i-1},\lambda_i)$ is a $\beta_i$-H\"older triangle, $\lambda_0=\gamma_1$ and $\lambda_p=\gamma_2$.
We proceed by induction on the number $p$ of pizza slices. The case $p=1$ follows from Lemma \ref{lem:pizzaslice}.
If $p>1$ then $tord(\lambda_1,T')>\beta$, otherwise $\{T_i\}$ would not be a minimal pizza decomposition.
It follows from Lemma \ref{lem:tord} applied to $T_1$ that there is an arc $\theta_1\subset T'$ such that
$tord(\gamma'_1,\theta)=\beta_1$ and conditions of Proposition \ref{map-graph} are satisfied for $T_1$ and $T'_1=T(\gamma'_1,\theta_1)$.

Let $\check T=T(\lambda_1,\gamma_2)$ and $\check T'=T(\theta_1,\gamma'_2)$. Since $tord(\lambda_1,T')>\beta$, $\check T$ and $\check T'$
have the same exponents (see Remark \ref{rem:beta}).
The same arguments as in the proof of Proposition \ref{symmetric-pizza} show that $tord(\gamma,\check T')=tord(\gamma,T')$ for any arc $\gamma\subset \check T$, thus $\{T_i\}_{i=2}^p$ is a minimal pizza decomposition of $\check T$ associated with the function $\check f(x)=dist(x,\check T')$.
By the inductional hypothesis, H\"older triangles $\check T$ and $\check T'$ satisfy conditions of Proposition \ref{map-graph}.
Proposition \ref{symmetric-pizza} implies that $T=T_1\cup \check T$ and $T'=\T'_1\cup\check T'$ satisfy conditions of Proposition \ref{map-graph}.
The existence of a mapping $h:T\to T'$ satisfying condition 5 of Proposition \ref{map-graph} implies that $Q_f(T)=Q_g(T')$ and
 $\mu_{T,f}\equiv\mu_{T',g}$.
\end{proof}

\section{The $\sigma\tau$-pizza invariant.}\label{the-invariant}

If two normally embedded H\"older triangles $T=T(\gamma_1,\gamma_2)$ and $T'=T(\gamma'_1,\gamma'_2)$ satisfying condition (\ref{tord-tord}) are not elementary with respect to the distance functions, then $T\cup T'$ may be not outer bi-Lipschitz equivalent to the union of $T$ and a graph of a function defined on $T$ (see Fig.~\ref{fig:non-elementary}).
In any case, a minimal pizza on $T$ associated with the function $f(x)=dist(x,T')$,
and a minimal pizza on $T'$ associated with the function $g(x')=dist(x',T)$, are outer Lipschitz invariants of the pair $(T,T')$.
The following example shows that two pairs $(T,T')$ and $(\tilde T,\tilde T')$ of normally embedded triangles satisfying condition (\ref{tord-tord})
may be not outer bi-Lipschitz equivalent even when the minimal pizzas on $T$ and $T'$ are equivalent to the minimal pizzas on $\tilde T$ and $\tilde T'$ respectively.

\begin{example}\label{ladder}\normalfont
The links of two normally embedded H\"older triangles $T=T(\gamma_1,\gamma_2)$ and $T'=T(\gamma'_1,\gamma'_2)$ are shown in Fig.~\ref{fig:one-to-one}.
Triangle $T$ is partitioned by the arcs $\lambda_1$, $\lambda_2$, $\lambda_3$, $\lambda_4$ into H\"older triangles $T_1=T(\gamma_1,\lambda_1)$, $T_2=T(\lambda_1,\lambda_2)$, $T_3=T(\lambda_2,\lambda_3)$, $T_4=T(\lambda_3,\lambda_4)$, $T_5=T(\lambda_4,\gamma_2)$ with exponents $\mu_2$, $q_2$, $\mu_1$, $q_2$, $\mu_2$, respectively,
and triangle $T'$ is partitioned by the arcs $\lambda'_1$, $\lambda'_2$, $\lambda'_3$, $\lambda'_4$ into H\"older triangles $T'_1=T(\gamma'_1,\lambda'_1)$, $T'_2=T(\lambda'_1,\lambda'_2)$, $T'_3=T(\lambda'_2,\lambda'_3)$, $T'_4=T(\lambda'_3,\lambda'_4)$, $T'_5=T(\lambda'_4,\gamma'_2)$ with exponents $\mu_2$, $q_2$, $\mu_1$, $q_2$, $\mu_2$, respectively, so that the following holds:\newline
$tord(\gamma,T')=q_2$ for any arc $\gamma\subset T_1$, $tord(\gamma,T')=q_1$ for any arc $\gamma\subset T_3$, $tord(\gamma,T')=q_2$ for any arc $\gamma\subset T_5$;
$tord(\gamma,T')=tord(\gamma,\lambda'_2)$ for any arc $\gamma\subset T_2$, $tord(\gamma,T')=tord(\gamma,\lambda'_3)$ for any arc $\gamma\subset T_4$;
$tord(\gamma',T)=q_2$ for any arc $\gamma'\subset T'_1$, $tord(\gamma',T)=q_1$ for any arc $\gamma'\subset T'_3$, $tord(\gamma',T)=q_2$ for any arc $\gamma'\subset T'_5$;
$tord(\gamma',T)=tord(\gamma',\lambda_2)$ for any arc $\gamma'\subset T'_2$, $tord(\gamma',T)=tord(\gamma',\lambda_3)$ for any arc $\gamma'\subset T'_4$.

In particular, $T$ and $T'$ satisfy condition (\ref{tord-tord}):
$$tord(\gamma_1,T')=tord(\gamma_1,\gamma'_1)=tord(\gamma'_1,T)=tord(\gamma_2,T')=tord(\gamma_2,\gamma'_2)=tord(\gamma'_2,T)=q_2.$$
Assuming $q_1>\mu_1\ge q_2>\mu_2$, the arcs $\lambda_1,\ldots,\lambda_4$ define a minimal pizza decomposition of $T$ associated with the function $f(x)=dist(x,T')$.

Although $T$ is not elementary with respect to $f(x)$, the union $T\cup T'$
is outer bi-Lipschitz equivalent to the union of $T$ and the graph of $f(x)$.
In particular, a minimal pizza on $T'$ associated with the function $g(x')=dist(x',T)$ is equivalent to a minimal pizza on $T$ associated
with the function $f(x)$.

H\"older triangles $\tilde T$ and $\tilde T'$ in Fig.~\ref{fig:one-to-two} are also normally embedded and satisfy condition (\ref{tord-tord}):
$$tord(\gamma_1,\tilde T')=tord(\gamma_1,\gamma'_1)=tord(\gamma'_1,\tilde T)=tord(\gamma_2,\tilde T')=tord(\gamma_2,\gamma'_2)=tord(\gamma'_2,\tilde T)=q_2.$$
Triangle $\tilde T$ is partitioned by the arcs $\lambda_1$, $\lambda_2$, $\lambda_3$, $\lambda_4$ into H\"older triangles $\tilde T_1=T(\gamma_1,\lambda_1)$, $\tilde T_2=T(\lambda_1,\lambda_2)$, $\tilde T_3=T(\lambda_2,\lambda_3)$, $\tilde T_4=T(\lambda_3,\lambda_4)$, $\tilde T_5=T(\lambda_4,\gamma_2)$,
and triangle $\tilde T'$ is partitioned by the arcs $\lambda'_1$, $\lambda'_2$, $\lambda'_3$, $\lambda'_4$ into H\"older triangles $\tilde T'_1=T(\gamma'_1,\lambda'_1)$, $\tilde T'_2=T(\lambda'_1,\lambda'_2)$, $\tilde T'_3=T(\lambda'_2,\lambda'_3)$, $\tilde T'_4=T(\lambda'_3,\lambda'_4)$, $\tilde T'_5=T(\lambda'_4,\gamma'_2)$.
Conditions satisfied by these triangles are the same as for those in Fig.~\ref{fig:one-to-one}, except
$tord(\gamma,\tilde T')=tord(\gamma,\lambda'_3)$ for any arc $\gamma\subset \tilde T_2$, $tord(\gamma,\tilde T')=tord(\gamma,\lambda'_2)$ for any arc $\gamma\subset \tilde T_4$; $tord(\gamma',\tilde T)=tord(\gamma',\lambda_3)$ for any arc $\gamma'\subset \tilde T'_2$, $tord(\gamma',\tilde T)=tord(\gamma',\lambda_2)$ for any arc $\gamma'\subset \tilde T'_4$.
Assuming $q_1>\mu_1\ge q_2>\mu_2$, the arcs $\lambda_1,\ldots,\lambda_4$ define a minimal pizza decomposition of $\tilde T$ associated with the function $\tilde f(x)=dist(x,\tilde T')$.

One can show that minimal pizzas on $T$ and $\tilde T$ are equivalent,
and a minimal pizza on $\tilde T'$ associated with the function $\tilde g(x')=dist(x',\tilde T)$ is equivalent to a minimal pizza on $\tilde T$ associated with the function $\tilde f(x)$. Thus minimal pizzas on $T'$ and $\tilde T'$ are also equivalent.

However, the union $\tilde T\cup\tilde T'$ is not outer bi-Lipschitz equivalent to the union of $\tilde T$ and the graph of $\tilde f(x)$.
In particular, $\tilde T\cup\tilde T'$ is not outer bi-Lipschitz equivalent to $T\cup T'$.
\end{example}

\begin{figure}
\centering
\includegraphics[width=5in]{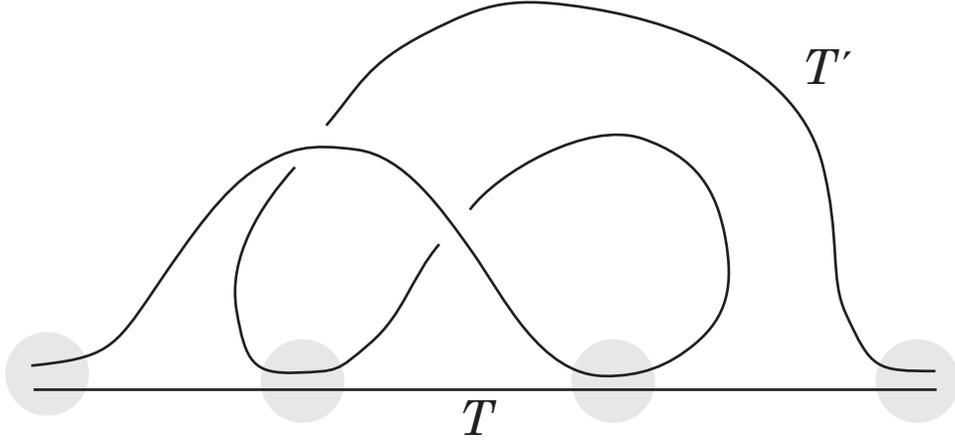}
\caption{Two normally embedded $\beta$-H\"older triangles, not elementary with respect to the distance functions.
Shaded disks indicate zones with the tangency order higher than $\beta$.}\label{fig:non-elementary}
\end{figure}

\begin{figure}
\centering
\includegraphics[width=5in]{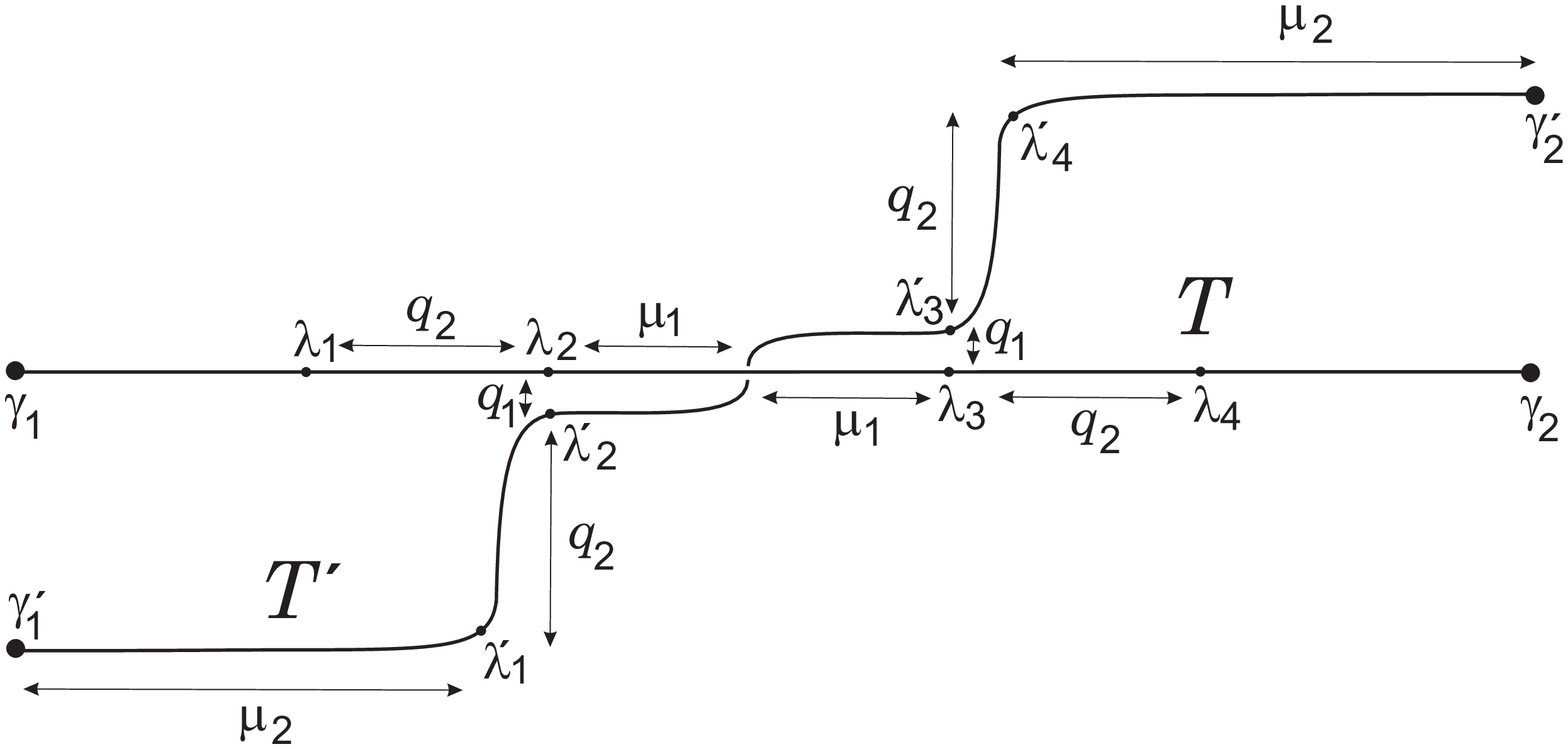}
\caption{Two normally embedded H\"older triangles $T$ and $T'$ in Example \ref{ladder}.}\label{fig:one-to-one}
\end{figure}

\bigskip
\begin{figure}
\centering
\includegraphics[width=5in]{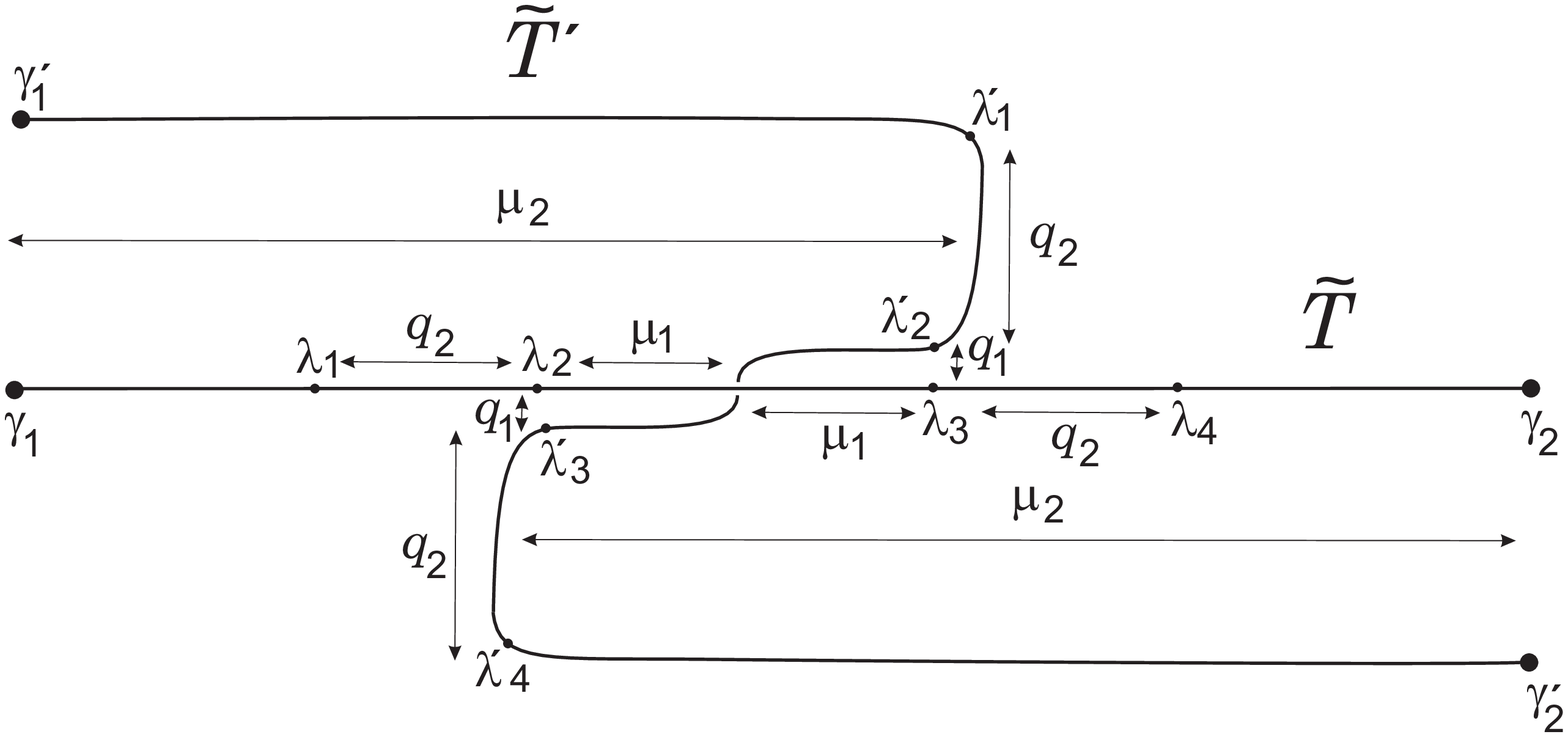}
\caption{Two normally embedded H\"older triangles $\tilde T$ and $\tilde T'$ in Example \ref{ladder}.}\label{fig:one-to-two}
\end{figure}

\begin{definition}\label{maxmin}\normalfont
Let $T=T(\gamma_1,\gamma_2)$ and $T'=T(\gamma'_1,\gamma'_2)$ be normally embedded $\beta$-H\"older triangles,
oriented from $\gamma_1$ to $\gamma_2$ and from $\gamma'_1$ to $\gamma'_2$ respectively, satisfying condition (\ref{tord-tord}).
Let $f(x)=dist(x,T')$ and $g(x')=dist(x',T)$ be the distance functions defined on $T$ and $T'$ respectively.
Let $D_\ell\subset V(T)$, for $\ell=0,\ldots,p$, be the pizza zones of a minimal pizza on $T$ associated with $f(x)$,
ordered according to the orientation of $T$, and let $q_\ell=tord(D_\ell,T')=ord_{\gamma}f$ for any arc $\gamma\in D_\ell$.
A zone $D_\ell$ is called a \emph{maximal exponent zone} for $f(x)$ (or simply a \emph{maximum zone}) if either $0<\ell<p$ and $q_\ell\ge\max(q_{\ell-1},q_{\ell+1})$, or $\ell=0$ and $\beta<q_0\ge q_1$, or $\ell=p$ and $\beta<q_p\ge q_{p-1}$.
If a zone $D_\ell$ is not a maximum zone, it is called a \emph{minimal exponent zone} for $f(x)$ (or simply a \emph{minimum zone})
if either $0<\ell<p$ and $q_\ell\le\min(q_{\ell-1},q_{\ell+1})$, or $\ell=0$ and $q_0\le q_1$, or $\ell=p$ and $q_p\le q_{p-1}$.
Maximum and minimum pizza zones $D'_{\ell'}\subset V(T')$ for a minimal pizza on $T'$ associated with $g(x')$
 are defined similarly, exchanging $T$ and $T'$.
\end{definition}

\begin{remark}\label{altmaxmin}\normalfont
Each of the singular pizza zones $D_0=\{\gamma_1\}$ and $D_p=\{\gamma_2\}$ is either a maximum or a minimum zone.
When $p=1$ and $q_0=q_1>\beta$, both $D_0$ and $D_1$ are maximum zones. When $p=1$ and $q_0=q_1\le\beta$, both $D_0$ and $D_1$ are minimum zones.
If $p>1$ and $0<\ell<p$, then $q_\ell>\min(q_{\ell-1},q_{\ell+1})$ if $D_\ell$ is a maximum zone, $q_\ell<\max(q_{\ell-1},q_{\ell+1})$ if $D_\ell$ is a minimum zone.
\end{remark}

\begin{proposition}\label{sigma}
Let $T=T(\gamma_1,\gamma_2)$ and $T'=T(\gamma'_1,\gamma'_2)$ be two normally embedded H\"older triangles,
oriented from $\gamma_1$ to $\gamma_2$ and from $\gamma'_1$ to $\gamma'_2$ respectively, satisfying condition (\ref{tord-tord}).
Let $\{M_i\}_{i=1}^m$ and $\{M'_{i'}\}_{i'=1}^{m'}$ be the maximum zones in $V(T)$ and $V(T')$
for the distance functions $f(x)=dist(x,T')$ and $g(x')=dist(x',T)$ respectively, ordered according to the orientations of $T$ and $T'$.
Let $\bar q_i=tord(M_i,T')$ and $\bar q'_{i'}=tord(M'_{i'},T)$.
Then $m'=m$, and there is a canonical one-to-one correspondence $i'=\sigma(i)$ between the zones $M_i$ and $M'_{i'}$,
such that $\mu(M'_{i'})=\mu(M_i)$ and $tord(M_i,M'_{i'})=\bar q_i=\bar q'_{i'}$.
If $\{\gamma_1\}=M_1$ is a maximum zone then $M'_1=\{\gamma'_1\}$ and $\sigma(1)=1$. If $\{\gamma_2\}=M_m$ is a maximum zone then
$M'_m=\{\gamma'_2\}$ and $\sigma(m)=m$.
\end{proposition}

\begin{proof}
The case $p=1$ follows from Lemma \ref{lem:pizzaslice}, thus we may assume $p>1$.

Let us choose any arcs $\lambda_i\in D_i$, so that $\{T_i=T(\lambda_{i-1},\lambda_i)\}$ is a minimal pizza on $T$ associated with the function $f(x)$.
Let $q_i=ord_{\lambda_i}f$ for $i=0,\ldots,p$, and let $\beta_i=tord(\lambda_{i-1},\lambda_i)$ be the exponent of a pizza slice $T_i$, for $i=1,\ldots,p$.

Consider first the case when $\{\gamma_1\}=\{\lambda_0\}$ is a maximum zone for $f(x)$.
Then $q_0=tord(\gamma_1,\gamma'_1)$ and $q_1=ord_{\lambda_1}f\le q_0$.
If $q_1>\beta_1$, it follows from Lemma \ref{lem:pizzaslice} that, for any arc $\lambda'\subset T'$ such that $tord(\lambda_1,\lambda')=q_1$,
H\"older triangles $T_1$ and $T'_1=T(\gamma'_1,\lambda')$ satisfy (\ref{tord-tord}) and conditions of Proposition \ref{map-graph}.
If $\gamma'_1$ is not a maximum zone for $g(x')$ then, for any arc $\lambda'_1$ such that $T'_1=T(\gamma'_1,\lambda'_1)$ is a pizza slice for a minimal pizza associated with $g(x')$, we have $q'_1=tord(\lambda'_1,T)>q_0$.
Let $\lambda\subset T$ be any arc such that $tord(\lambda,\lambda'_1)=q'_1$.
Then Lemma \ref{lem:pizzaslice} applied to $T'_1$ and $\bar T=T(\gamma_1,\lambda)$ implies that $\bar T$ is a pizza slice for $f(x)$.
Since $tord(\gamma,T')\le q_0$ for any arc $\gamma\subset T_1$, we have $T_1\subset\bar T$, a contradiction with $T_1$ being a pizza slice for a minimal pizza associated with $f(x)$.

Similarly, if $\{\gamma_2\}$ is a maximum zone for $f(x)$ then $\{\gamma'_2\}$ is a maximum zone for $g(x')$.

Suppose next that $p>1$ and $M_i$ is a maximum zone for $f(x)$, where $0<i<p$. Let $\lambda'\subset T'$ be any arc such that $tord(\lambda_i,\lambda')=q_i$. We are going to show that $\lambda'$ belongs to a maximum zone $D'$ for a minimal pizza on $T'$ associated with $g(x')$.

Note first that, if $\lambda'$ belongs to a pizza zone $D'$ for a minimal pizza on $T'$ associated with $g(x')$,
the same arguments as those for $\gamma_1$ and $\gamma'_1$ show that $D'$ is a maximum zone for $g(x')$.

Suppose that $\lambda'$ does not belong to a pizza zone. Let $D'_{j-1}$ and $D'_j$ be two adjacent pizza zones for
a minimal pizza on $T'$ associated with $g(x')$, $\lambda'_{j-1}\in D'_{j-1}$ and $\lambda'_j\in D'_j$, $\lambda'\subset T'_j=T(\lambda'_{j-1},\lambda'_j)$ and $\lambda'\notin D'_{j-1}\cup D'_j$.
The same arguments as those for $\gamma_1$ and $\gamma'_1$ show that $tord(\lambda'_{j-1},T)\le q_i$ and $tord(\lambda'_j,T)\le q_i$.
Since $T'_j$ is a pizza slice, at least one of these inequalities is an equality, say $tord(\lambda'_j,T)=q_i$.
Let $\mu'_j=\nu(\lambda'_j)\le q_i$ be the order of the pizza zone $D'_j$.
Since any arc $\gamma'\subset T'_j$ such that $tord(\gamma',\lambda'_j)\ge\nu(\lambda'_j)$ belongs to $D'_j$ and $\lambda'\notin D'_j$, we have $tord(\lambda',\lambda'_j)<\mu'_j\le q_i$. Thus $\bar T'=T(\lambda',\lambda'_j)$ is a $\bar\beta$-H\"older triangle, where $\bar\beta<q_i$,
such that $tord(\gamma',T)=q_i$ for any arc $\gamma'\subset \bar T'$.
\end{proof}

\begin{definition}\label{characteristic}
\emph{Let $T=T(\gamma_1,\gamma_2)$ and $T'=T(\gamma'_1,\gamma'_2)$ be two normally embedded H\"older triangles,
oriented from $\gamma_1$ to $\gamma_2$ and from $\gamma'_1$ to $\gamma'_2$ respectively, satisfying condition (\ref{tord-tord}).
Let $\{M_i\}_{i=1}^m$ and $\{M'_{i'}\}_{i'=1}^{m'}$ be the maximum zones in $V(T)$ and $V(T')$
for the functions $f(x)=dist(x,T')$ and $g(x')=dist(x',T)$ respectively, ordered according to the orientations of $T$ and $T'$.
According to Proposition \ref{sigma}, we have $m'=m$, and there is a canonical permutation $\sigma$ of the set $\{1,\ldots,m\}$,
the \emph{characteristic permutation} of the pair $T$ and $T'$,
such that $tord(M_i, M'_{\sigma(i)})=tord(M_i, T')=tord(M'_{\sigma(i)},T)$.}
\end{definition}

\begin{definition}\label{def:pizzaslicezone-transversal}\normalfont
Let $T=T(\gamma_1,\gamma_2)$ and $T'=T(\gamma'_1,\gamma'_2)$ be two normally embedded H\"older triangles,
oriented from $\gamma_1$ to $\gamma_2$ and from $\gamma'_1$ to $\gamma'_2$ respectively, satisfying condition (\ref{tord-tord}).
Let $D_\ell\subset V(T)$, for $\ell=0,\ldots,p$, be the pizza zones of a minimal pizza $\{T_\ell=T(\lambda_{\ell-1},\lambda_\ell)\}$ on $T$ associated with the distance function $f(x)=dist(x,T')$, ordered according to the orientation of $T$.
For $\ell=1,\ldots,p$, let $Y_\ell=D_{\ell-1}\cup D_\ell\cup V(T_\ell)$ be the maximal pizza slice zones in $V(T)$ associated with $f$ (see Corollary \ref{zone-slice}).
Let $Q_\ell=Q_f(Y_\ell)=[q_{\ell-1},q_\ell]$, where $q_\ell=ord_{\lambda_\ell} f$, and $\mu_\ell=\mu_{Y_\ell,f}:Q_\ell\to\F\cup\{\infty\}$ be the corresponding exponent intervals and affine width functions (see Definition \ref{def:width function}).
We say that a zone $Y_\ell$, and a pizza slice $T_\ell=T(\lambda_{\ell-1},\lambda_\ell)$ where $\lambda_{\ell-1}\in D_{\ell-1}$ and
$\lambda_\ell\in D_\ell$, is \emph{transversal} if $\mu_\ell(q)\equiv q$, and \emph{non-transversal} otherwise.
\end{definition}

\begin{proposition}\label{pizzaslice-oriented}
Let $T=T(\gamma_1,\gamma_2)$ and $T'=T(\gamma'_1,\gamma'_2)$ be two normally embedded H\"older triangles,
oriented from $\gamma_1$ to $\gamma_2$ and from $\gamma'_1$ to $\gamma'_2$ respectively, satisfying condition (\ref{tord-tord}).
Let $D_\ell$, $Y_\ell$, $Q_\ell=[q_{\ell-1},q_\ell]$ and $\mu_\ell$ be as in Definition \ref{def:pizzaslicezone-transversal}.
Let $D'_{\ell'}$, for $\ell'=0,\ldots,p'$, be the pizza zones of a minimal pizza
on $T'$ associated with $g(x')=dist(x',T)$, ordered according to the orientation of $T'$.
Let $Y'_{\ell'}\subset V(T')$, $Q'_{\ell'}=Q_g(Y'_{\ell'})=[q_{\ell'-1},q_{\ell'}]\subset\F\cup\{\infty\}$ and $\mu'_{\ell'}:Q'_{\ell'}\to\F\cup\{\infty\}$
be the corresponding maximal pizza slice zones, exponent intervals and affine width functions.
Then, for each index $\ell$ such that the pizza slice zone $Y_\ell$ is non-transversal, there is a unique index $\ell'=\tau(\ell)$ such that
$Q'_{\ell'}=Q_\ell$, $\mu'_{\ell'}\equiv\mu_\ell$ and one of the following two conditions holds:
\begin{multline}\label{tord-plus}
tord(D_\ell,T') = tord(D_\ell,D'_{\ell'}) = tord(D'_{\ell'},T), \\
tord(D_{\ell-1},T') = tord(D_{\ell-1},D'_{\ell'-1}) = tord(D'_{\ell'-1},T);
\end{multline}
\begin{multline}\label{tord-minus}
tord(D_\ell,T') = tord(D_\ell,D'_{\ell'-1}) = tord(D'_{\ell'-1},T), \\
tord(D_{\ell-1},T') = tord(D_{\ell-1},D'_{\ell'}) = tord(D'_{\ell'},T).
\end{multline}
\end{proposition}

\begin{proof} Let $Y_\ell\subset V(T)$ be a non-transversal maximal pizza slice zone for a minimal pizza associated with $f$.
For each $q\in Q_\ell$ let $Z_q\subset Y_\ell$ be the maximal $q$-order zone for $f$.
If $Q_\ell=\{q_\ell\}$ is a point, then $q_\ell>\mu_\ell$ since $Y_\ell$ is non-transversal.
It follows from Proposition \ref{long-q-zone} and Corollary \ref{correspondance}
that there is a unique maximal $q_\ell$-order zone $Z'\subset V(T')$ for $g$, of order $\mu_\ell$,
containing all arcs $\gamma'\subset V(T')$ such that $tord(\gamma',Y_\ell)=ord_{\gamma'}g=q_\ell$.
Let us show that $Z'$ is a maximal pizza slice zone for $g$.
Let $Z'\subset Y'_{\ell'}$ where $Y'_{\ell'}$ is a maximal pizza slice zone for a minimal pizza associated with $g$, of order $\mu'_{\ell'}$.
If $Z'\ne Y'_{\ell'}$ then either $Q'_{\ell'}=\{q_\ell\}$ is a point but $\mu_\ell>\mu'_{\ell'}$ or $Q'_{\ell'}$ is not a point.

In the first case, there is a $\beta'$-H\"older triangle $\tilde T'=T(\tilde\gamma'_1,\tilde\gamma'_2)$ such that $\beta'<\mu_\ell$,
$V(\tilde T')\subset Y'_{\ell'}$ and $Z'\cap V(T')$ is a $\mu_\ell$-zone. Let $\tilde T=T(\tilde\gamma_1,\tilde\gamma_2)\subset T$
be a $\beta'$-H\"older triangle, where $\tilde\gamma_1$ and $\tilde\gamma_2$ are two arcs in $T$ such that $tord(\tilde\gamma_1,\tilde\gamma'_1)=tord(\tilde\gamma_2,\tilde\gamma'_2)=q_\ell$.
Since $Y'_{\ell'}$ is a pizza slice zone, $\tilde T'$ is elementary with respect to $g$, and the pair $(\tilde T,\tilde T')$ satisfies (\ref{tord-tord}).
It follows from Theorem \ref{theorem:very-elementary} applied to $\tilde T'$ that $\tilde T$ is elementary with respect to $f$, $Q_f(\tilde T)=\{q_\ell\}$ and $Z\cap V(\tilde T)$ is a $\mu_\ell$-zone.
Since $\beta'<\mu_\ell$, this contradicts the assumption that $Y_\ell$ is a minimal pizza slice zone.

The arguments for the second case, when $Q_\ell$ is a point but $Q'_{\ell'}$ is not a point,
are similar: one can find a H\"older triangle $\tilde T'\subset T'$ such that $Q_g(\tilde T')=Q'_{\ell'}$ is not a point, $\tilde T'$ is elementary with respect to $g$, and $V(\tilde T')\cap Z'$ is a $\mu_\ell$-zone.
Then there is a H\"older triangle $\tilde T\subset T$ such that the pair $(\tilde T,\tilde T')$ satisfies (\ref{tord-tord}).
Theorem \ref{theorem:very-elementary} applied to $\tilde T'$ implies that $Q_f(\tilde T)$ is not a point, while the width function of $\tilde T$ is affine, a contradiction with the assumption that $Y_\ell$ is a minimal pizza slice zone.

Suppose now that $Q_\ell=[q_{\ell-1},q_\ell]$ is not a point. Then Proposition \ref{long-q-zone} and Corollary \ref{correspondance}
applies to each $q$-order zone $Z_q\subset Y_\ell$ for $f$ when $q\in\dot Q_\ell=(q_{\ell-1},q_\ell)$, but may be not applicable when $q=q_{\ell-1}$ or $q=q_\ell$ if $\mu_\ell(q)=q$. For $q\in\dot Q_\ell$, let $Z'_q\subset V(T')$ be the $q$-order zone for $g$, of order $\mu_\ell(q)$, corresponding to $Z_q$.
Then $Z=\bigcup_{q\in\dot Q_\ell} Z_q$ is a pizza slice zone for $f$, and $Z'=\bigcup_{q\in\dot Q_\ell} Z'_q$ is a pizza slice zone for $g$.
Let $\subset Y'_{\ell'}\supset\dot Z'$ be the maximal pizza slice zone for a minimal pizza associated with $g$, of order $\mu'_{\ell'}$.
The same arguments as above show that $Q'_{\ell'}=Q_\ell$ and $\mu'_{\ell'}\equiv\mu_{\ell}$.

Note that the pairs of zones $(Z_q,Z'_q)$ are either all positively oriented or all negatively oriented (see Definition \ref{def:zones-oriented}).
Accordingly, either (\ref{tord-plus}) or (\ref{tord-minus}) holds for the pairs of maximal pizza slice zones $(Y_\ell,Y'_{\ell'})$.
\end{proof}

\begin{definition}\label{def:tau}\normalfont
Let $T=T(\gamma_1,\gamma_2)$ and $T'=T(\gamma'_1,\gamma'_2)$ be two normally embedded H\"older triangles,
oriented from $\gamma_1$ to $\gamma_2$ and from $\gamma'_1$ to $\gamma'_2$ respectively, satisfying condition (\ref{tord-tord}).
Let $\{T_\ell\}$ and $\{T'_{\ell'}\}$ be minimal pizzas on $T$ and $T'$ for the distance functions $f(x)=dist(x,T')$ and $g(x')=dist(x',T)$ respectively, ordered according to the orientations of $T$ and $T'$.
Then, according to Proposition \ref{pizzaslice-oriented}, there is a canonical one-to-one correspondence $\ell'=\tau(\ell)$ between the sets
of non-transversal pizza slices $T_\ell$ for a minimal pizza on $T$ associated with $f(x)=dist(x,T')$, ordered according to the orientation of $T$, and the set of non-transversal pizza slices $T'_{\ell'}$ for a minimal pizza on $T'$ associated with $g(x')=dist(x',T)$, ordered according to the orientation of $T'$.
This defines a \emph{characteristic correspondence} $\tau$ between the sets of non-transversal pizza slices of $T$ and $T'$.
In particular, these two sets have the same number of elements.
We say that a pair of non-transversal pizza slice zones $Y_\ell$ and $Y'_{\ell'}$ where $\ell'=\tau(\ell)$, and a pair of non-transversal pizza slices $T_\ell$ and $T'_{\ell'}$, is \emph{positively oriented} if (\ref{tord-plus}) holds and \emph{negatively oriented} otherwise (see Definition \ref{def:pizzaslice-oriented}).
Thus $\tau$ is a \emph{signed} correspondence, with the signs $+$ and $-$ assigned to the positively and negatively oriented pairs of non-transversal pizza slice zones.
\end{definition}

\begin{remark}\label{rem:tau}\normalfont
For each pair $(T_\ell,T'_{\ell'})$ of non-transversal pizza slices, where $\ell'=\tau(\ell)$, the signed correspondence $\tau$
defines a correspondence between the two pizza zones $D_{\ell-1}$ and $D_\ell$ and the two pizza zones $D'_{\ell'-1}$ and $D'_{\ell'}$,
in the same (resp., opposite) order if the pair is positively (resp., negatively) oriented.
This correspondence between a subset of pizza zones of $T$ and a subset of pizza zones of $T'$ may be not one-to-one:
a pizza zone of $T$ common to two non-transversal pizza slices may correspond to two different pizza zones of $T'$,
and two different pizza zones of $T$ may correspond to the same pizza zone of $T'$ (see Fig.~\ref{fig:oneloop}).
However, it is one-to-one on the set of those pizza zones which are also maximum zones (see Proposition \ref{sigma=tau}).
\end{remark}

\begin{figure}
\centering
\includegraphics[width=5.5in]{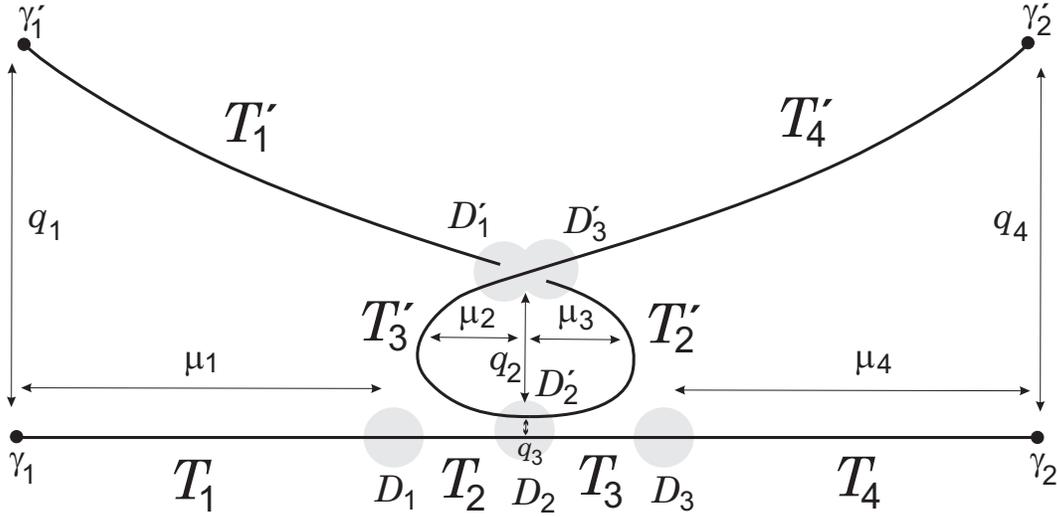}
\caption{Two normally embedded H\"older triangles in Remark \ref{rem:tau}.
Shaded disks indicate pizza zones of minimal pizzas on $T$ and $T'$.
Assuming $q_1=q_4>\mu_1=\mu_4$ and $q_2=\mu_2=\mu_3<q_3$, there are four non-transversal pairs of pizza slices: $(T_1,T'_1)$, $(T_2,T'_3)$, $(T_3,T'_2)$, $(T_4,T'_4)$.
The correspondence $\tau(1)$ maps $D_1$ to $D'_1$, while $\tau(2)$ maps $D_1$ to $D'_3$.}\label{fig:oneloop}
\end{figure}

\begin{proposition}\label{sigma=tau}
Let $T=T(\gamma_1,\gamma_2)$ and $T'=T(\gamma'_1,\gamma'_2)$ be normally embedded H\"older triangles,
oriented from $\gamma_1$ to $\gamma_2$ and from $\gamma'_1$ to $\gamma'_2$ respectively, satisfying condition (\ref{tord-tord}).
Let $\{T_\ell\}$ and $\{T'_{\ell'}\}$ be minimal pizzas on $T$ and $T'$ for the distance functions $f(x)=dist(x,T')$ and $g(x')=dist(x',T)$ respectively, ordered according to the orientations of $T$ and $T'$.
Let $(T_\ell,T'_{\ell'})$, where $\ell'=\tau(\ell)$, be a pair of non-transversal pizza slices
such that one of the pizza zones of $T_\ell$, say $D=D_\ell$, is a maximum zone $M_i\subset V(T)$. Then the corresponding pizza zone $D'$ of $T'_{\ell'}$ (either $D'=D'_{\ell'}$ for a positively oriented pair $(T_\ell,T'_{\ell'})$ or $D'=D'_{\ell'-1}$ for a negatively oriented pair) is a maximum zone $M'_{i'}\subset V(T')$, where $i'=\sigma(i)$.
\end{proposition}

\begin{proof} If $D$ is a boundary arc of $T$ then the statement follows from Proposition \ref{sigma},
since $D'$ is also a boundary arc of $T'$ and a (singular) maximum zone.

If $D=D_\ell$ is not a boundary arc, and both maximal pizza slice zones $Y_\ell$ and $Y_{\ell+1}$ containing $D_\ell$ are non-transversal, then Proposition \ref{pizzaslice-oriented} implies that the corresponding zones in $V(T')$ are either $Y'_{\ell'}$ and $Y'_{\ell'+1}$ (if $(Y_\ell,Y'_{\ell'})$ is a positively oriented pair) or  $Y'_{\ell'}$ and $Y'_{\ell'-1}$ (if $(Y_\ell,Y'_{\ell'})$ is a negatively oriented pair). In both cases, Proposition \ref{pizzaslice-oriented} implies that $D'$ is a maximum zone such that $tord(D,D')=tord(D,T')=tord(D',T)=q_\ell$, thus $i'=\sigma(i)$.

Otherwise, if $Y_\ell$ is a non-transversal zone but $Y_{\ell+1}$ is transversal, Lemma \ref{lem:tord} implies that there are two $\tilde\beta$-H\"older triangles $\tilde T=T(\tilde\gamma_1,\tilde\gamma_2)\subset T_{\ell+1}$ and $\tilde T'=T(\tilde\gamma'_1,\tilde\gamma'_2)\subset T'$ satisfying (\ref{tord-tord}), where $\tilde\beta<q_\ell$,
such that $\tilde T'\cap T'_\ell=\{\tilde\gamma'_1\}$, $\tilde\gamma_1\in D$, $\tilde\gamma'_1\in D'$ and $tord(\tilde\gamma_1,\tilde\gamma'_1)=q_\ell$.
Theorem \ref{theorem:very-elementary} applied to $\tilde T$ and $\tilde T'$ implies that $D'$
is a maximum zone such that $tord(D,D')=tord(D,T')=tord(D',T)=q_\ell$, thus $i'=\sigma(i)$.
\end{proof}

\begin{proposition}\label{signed-Permutation}
Let $T=T(\gamma_1,\gamma_2)$ and $T'=T(\gamma'_1,\gamma'_2)$ be two normally embedded H\"older triangles,
oriented from $\gamma_1$ to $\gamma_2$ and from $\gamma'_1$ to $\gamma'_2$ respectively, satisfying condition (\ref{tord-tord}).
Then the sign assigned to each pair $(T_\ell,T'_{\ell'})$ of non-transversal pizza slices such that $\ell'=\tau(\ell)$
is completely determined by the minimal pizzas $\{T_\ell\}$ and $\{T'_{\ell'}\}$, the characteristic permutation $\sigma$ and the
characteristic correspondence $\tau$.
\end{proposition}

\begin{proof}
Let $Y_\ell\subset V(T)$ and $Y'_{\ell'}\subset V(T')$ be two non-transversal pizza slice zones such that $\ell'=\tau(\ell)$.
According to Definition \ref{def:tau}, the pair $(Y_\ell, Y'_{\ell'})$ is positively oriented if (\ref{tord-plus}) holds
and negatively oriented if (\ref{tord-minus}) holds. If $Q_\ell=[q_\ell,q_{\ell+1}]$ is not a point then $q_\ell\ne q_{\ell+1}$, thus
$Q'_{\ell'}=[q'_{\ell'},q'_{\ell'+1}]$ is also not a point, and
the pair is positive when $q_\ell=q'_{\ell'}$ and negative otherwise. If $Q_\ell=\{q_\ell\}$ is a point then $\mu_\ell<q_\ell$,
and each of the pizza zones $D_\ell$ and $D_{\ell+1}$ is either a maximum or a minimum zone.
If, say, $D_\ell=M_i$ is a maximum zone, then the pair is positive when $D'_{\ell'}=M_{\sigma(i)}$ and negative otherwise.

The case when each of them contains a boundary arc is trivial, so we may assume that they are interior zones. If $Q_\ell=Q_{\ell'}$ is not a point then the sign is uniquely determined by the maxima of non-constant affine functions $\mu_\ell(q)\equiv\mu'_{\ell'}(q)$. If $Q_\ell=Q'_{\ell'}=\{q_\ell\}>\mu_\ell$ is a point, then the pizza zones $D_{\ell-1}$ and $D_\ell$ correspond to the pizza zones $D'_{\ell'-1}$ and $D'_{\ell'}$ in the same order if the sign is positive, and in the opposite order if the sign is negative. Note that each of these zones is either a maximum or a minimum zone, since on one side of each of them $q$ is constant. The correspondence between the pizza slice zones sends a maximum pizza zone in $V(T)$ to a maximum pizza zone in $V(T')$, and a minimum pizza zone in $V(T)$ to a minimum pizza zone in $V(T')$. If one of the pizza zones in $V(T)$ is a maximum zone and another is a minimum zone, then the same is true for the pizza zones in $V(T')$, and the correspondence is uniquely defined. If both pizza zones are maximum zones then the correspondence is defined by $\sigma$. If both pizza zones in $V(T)$ are minimum zones, since $q_\ell>\mu_\ell$, the two maximum pizza zones in $V(T)$ closest to $Y_\ell$ are mapped by $\sigma$ to  the two maximum pizza zones in $V(T')$ closest to $Y'_{\ell'}$ in the same order if the sign is positive and in the opposite order if the sign is negative: one can only get from one side of $Y_\ell$ to another side through a part of $T'$ where $q\le\mu_\ell$.
\end{proof}

\begin{definition}\label{sigma-pizza}
\emph{Let $T=T(\gamma_1,\gamma_2)$ and $T'=T(\gamma'_1,\gamma'_2)$ be two normally embedded H\"older triangles,
oriented from $\gamma_1$ to $\gamma_2$ and from $\gamma'_1$ to $\gamma'_2$ respectively, satisfying condition (\ref{tord-tord}).
Let $\{T_\ell\}$ and $\{T'_{\ell'}\}$ be minimal pizzas on $T$ and $T'$ for the distance functions $f(x)=dist(x,T')$ and $g(x')=dist(x',T)$ respectively, ordered according to the orientations of $T$ and $T'$.
A $\sigma\tau$-\emph{pizza} on $T\cup T'$ is a triplet consisting of the pair of minimal pizzas $\{T_\ell\}$ and $\{T'_{\ell'}\}$, the characteristic  permutation $\sigma$ of the maximum pizza zones in $V(T)$ and $V(T')$, and the characteristic correspondence $\tau$ of the non-transversal pizza slices of $T$ and $T'$.
Two $\sigma\tau$-pizzas $(\{T_\ell\},\,\{T'_{\ell'}\},\,\sigma_T,\,\tau_T)$ on $T\cup T'$ and $(\{S_\ell\},\,\{S'_{\ell'}\},\,\sigma_S,\,\tau_S)$ on $S\cup S'$ are \emph{combinatorially equivalent} if the pairs $(\{T_\ell\},\,\{T'_{\ell'}\})$ and $(\{S_\ell\},\,\{S'_{\ell'}\})$ are combinatorially equivalent, $\sigma_T=\sigma_S$ and $\tau_T=\tau_S$.}
\end{definition}

\begin{theorem}\label{invariant} Let $(T,T')$ and $(S,S')$ be two oriented pairs of normally embedded
H\"older triangles satisfying condition (\ref{tord-tord}).
If there is an orientation-preserving outer bi-Lipschitz homeomorphism $H:T\cup T'\to S\cup S'$ such that $H(T)=S$ and $H(T')=S'$, then the $\sigma\tau$-pizzas of the pairs $(T,T')$ and $(S,S')$ are combinatorially equivalent.
\end{theorem}

\begin{proof}
Let $f_T(x)=dist(x,T')$, $g_T(x')=dist(x',T)$, $f_S(y)=dist(y,S')$ and $g_S(y')=dist(y',S)$ be the distance functions
defined on $T,\,T',\,S$ and $S'$ respectively.
Let $M_i\subset V(T)$ and $M'_{i'}\subset V(T')$ be the maximum zones for the pair $(T,T')$,
and let $N_i\subset V(S)$ and $N'_{i'}\subset V(S')$ be the maximum zones for the pair $(S,S')$.

Since $H$ is an outer bi-Lipschitz homeomorphism, $f_T$ is Lipschitz contact equivalent to $f_S$, and $g_T$ is Lipschitz contact equivalent to $g_S$. Theorem \ref{pizza-theorem} implies that the corresponding pairs of minimal pizzas $(\{T_\ell\},\,\{T'_{\ell'}\})$ and $(\{S_\ell\},\,\{S'_{\ell'}\})$
are combinatorially equivalent.
Accordingly, $H$ maps each maximum zone $M_i$ to the maximum zone $N_i$, and each maximum zone $M'_{i'}$ to the maximum zone $N'_{i'}$.
A pair of maximum zones $(M_i, M'_{i'})$, where $i'=\sigma_T(i)$, is mapped to the pair of maximum zones $(N_i, N'_{i'})$ preserving the order of contact
between these zones. This implies that $i'=\sigma_S(i)$, thus the permutations $\sigma_T$ and $\sigma_S$ are equal.

Moreover, since $H$ preserves the tangency orders between arcs, it maps each maximal pizza slice zone $Y_\ell$ of a minimal pizza on $T$
associated with $f_T$ to the maximal pizza slice zone $Z_\ell$ of a minimal pizza on $S$ associated with $f_S$,
and each maximal pizza slice zone $Y'_{\ell'}$ of a minimal pizza on $T'$
associated with $g_T$ to the maximal pizza slice zone $Z'_{\ell'}$ of a minimal pizza on $S'$ associated with $g_S$,
with the corresponding width functions preserved.
Accordingly, if $(T_\ell,T'_{\ell'})$ is a non-transversal pair of pizza slices of minimal pizzas associated with $f_T$ and $g_T$,
where $\ell'=\tau_T(\ell)$,
then $(H(T_\ell),\,H(T'_{\ell'}))$ is a non-transversal pair of pizza slices of minimal pizzas associated with $f_S$ and $g_S$,
such that $V(H(T_\ell))\subset Z_\ell$ and $V(H(T'_{\ell'}))\subset Z'_{\ell'}$. This implies that $\ell'=\tau_S(\ell)$,
thus the correspondences $\tau_T$ and $\tau_S$ are equal.
Proposition \ref{signed-Permutation} implies that $\tau_T$ and $\tau_S$ are equal also as signed correspondences.
\end{proof}

The following conjecture states that, conversely, two pairs of normally embedded H\"older triangles satisfying condition (\ref{tord-tord})
with the same $\sigma\tau$-pizza invariant are outer bi-Lipschitz equivalent, thus the $\sigma\tau$-pizza is a complete combinatorial
invariant of an outer bi-Lipschitz equivalence class of pairs of normally embedded H\"older triangles.

\begin{conjecture}
 Let $(T,T')$ and $(S,S')$ be two ordered oriented pairs of normally embedded H\"older triangles satisfying condition (\ref{tord-tord}).
 If the $\sigma\tau$-pizza of the pair $(T,T')$ is combinatorially equivalent to the $\sigma\tau$-pizza of the pair $(S,S')$,
then there is an orientation-preserving outer bi-Lipschitz homeomorphism $H:T\cup T'\to S\cup S'$ such that $H(T)=S$ and $H(T')=S'$.
\end{conjecture}


\begin{thebibliography}{99}
\bibitem {birbrair1999} L. Birbrair, Local  bi-Lipschitz  classification  of  2-dimensional  semialgebraic  sets. Houston Journal of Mathematics, N3, vol.25, (1999), pp 453-472.
\bibitem{birbrair2000normal} L.~Birbrair, T.~Mostowski. {\it Normal embeddings of semialgebraic sets.} Michigan Math.~J., 47 (2000),
125--132
\bibitem{birbrair2014lipschitz} L.~Birbrair, A.~Fernandes, A.~Gabrielov, V.~Grandjean. {\it Lipschitz contact equivalence of function
    germs in {$\mathbb{R}^2$}.} Annali SNS Pisa, 17 (2017), 81--92, DOI 10.2422/2036-2145.201503{$\_$}014
\bibitem{Bi-Fer-Co-Ru} L.~Birbrair, A.~Fernandes, J.~Costa, M.~Ruas. {\it K-bi-Lipschitz equivalence of real function-germs.} Proceedings of the American Mathematical Society, Estados Unidos, v. 135, p. 1089-1095, 2007.
\bibitem{extention} L.~Birbrair, A.~Fernandes, Z.~Jelonek. {\it On the extension of bi-Lipschitz mappings.} Selecta Mathematica 27, Article number: 15 (2021) DOI : 10.1007/s00029-021-00625-6
\bibitem{birbrair-mendes} L. ~Birbrair and R.~Mendes.{\it Arc criterion of normal embedding.} In Singularities
and foliations. geometry, topology and applications, volume 222 of Springer Proc.
Math. Stat., pages 549–553. Springer, Cham, 2018.
\bibitem{GS} A.~Gabrielov, E.~Souza. {\it Lipschitz geometry and combinatorics of abnormal surface germs.}
Selecta Math., New Series, 28:1 (2022). DOI:10.1007/s00029-021-00716-4

\bibitem{kur-orr}K.~Kurdyka and P.~Orro, Distance géodésique sur un sous-analytique, Rev. Mat. Univ. Complut. Madrid 10 (1997), supplementary, 174–182.
\bibitem{Mostowski} T.~Mostowski. {\it Lipschitz equisingularity} Dissertationes Math. (Rozprawy Mat.), 243 (1985) 46 pp.
\bibitem{valette2005Lip} G.~Valette {\it Lipschitz triangulations.} Illinois ~J. Math. 49 (2005), issue ~3, 953--979

\end{thebibliography}
\end{document}